\documentclass[12pt]{amsart}

\usepackage{amsfonts,amsmath,amsthm,amssymb,latexsym,amsthm,newlfont,enumerate,color,tikz-cd}

\newif\ifslide

\theoremstyle{plain}
\newtheorem{theorem}{Theorem}[section]

\newtheorem{lemma}[theorem]{Lemma}

\newtheorem{proposition}[theorem]{Proposition}
\newtheorem{definition-lemma}[theorem]{Definition-Lemma}

\newtheorem{red-question}[theorem]{\textcolor{red}{Question}}

\theoremstyle{definition}
\newtheorem{definition}[theorem]{Definition}
\newtheorem{remark}[theorem]{Remark}

\newtheorem{example}[theorem]{Example}

\def\ideal#1.{I_{#1}}
\def\ring#1.{\mathcal {O}_{#1}}

\def\fring#1.{\hat{\mathcal {O}}_{#1}}
\def\proj#1.{\mathbb {P}(#1)}
\def\pr #1.{\mathbb {P}^{#1}}
\def\dpr #1.{\hat{\mathbb {P}}^{#1}}
\def\af #1.{\mathbb A^{#1}}
\def\Hz #1.{\mathbb F_{#1}}
\def\Hbz #1.{\overline{\mathbb F}_{#1}}
\def\fb#1.{\underset #1 {\times}}
\def\rest#1.{\underset {\ \ring #1.} \to \otimes}
\def\au#1.{\operatorname {Aut}\,(#1)}
\def\deg#1.{\operatorname {deg } (#1)}

\def\pic#1.{\operatorname {Pic}\,(#1)}
\def\pico#1.{\operatorname{Pic}^0(#1)}
\def\picg#1.{\operatorname {Pic}^G(#1)}
\def\ner#1.{NS (#1)}
\def\rdown#1.{\llcorner#1\lrcorner}
\def\rfdown#1.{\lfloor{#1}\rfloor}
\def\rup#1.{\ulcorner{#1}\urcorner}
\def\rcup#1.{\lceil{#1}\rceil}

\def\n1#1.{\operatorname {N_1}(#1)}  
\def\cn1#1.{\overline{\operatorname {N^1}(#1)}} 
\def\cone#1.{\operatorname {NE}(#1)}     
\def\ccone#1.{\overline{\operatorname {NE}}(#1)}
\def\none#1.{\operatorname {NF}(#1)}
\def\cnone#1.{\overline{\operatorname {NF}}(#1)}
\def\mone#1.{\operatorname {NM}(#1)} 
\def\cmone#1.{\overline{\operatorname {NM}}(#1)}

\def\coef#1.{\frac{(#1-1)}{#1}}
\def\vit#1.{D_{\langle #1 \rangle}}
\def\mm#1.{\overline {M}_{0,#1}}
\def\H1#1.{H^1(#1,{\ring #1.})}
\def\ac#1.{\overline {\mathbb F}_{#1}}

\def\adj#1.{\frac {#1-1}{#1}}
\def\spn#1.{\overline{#1}}
\def\pek#1.#2.{\Cal P^{#1}(#2)}
\def\plk#1.#2.{\Cal P^{\leq #1}(#2)}
\def\ev#1.{\operatorname{ev_{#1}}}
\def\ilist#1.{{#1}_1,{#1}_2,\dots}
\def\bminv#1.{(\nu_1,s_1;\nu_2,s_2;\dots ;\nu_{#1},s_{#1};\nu_{r+1})}
\def\zinv#1.{(\nu_1,s_1;\nu_2,s_2;\dots ;\nu_{#1},s_{#1};0)}
\def\iinv#1.{(\nu_1,s_1;\nu_2,s_2;\dots ;\nu_{#1},s_{#1};\infty)}

\def\scr #1.{\mathcal #1}

\def\llist#1.#2.{{#1}_1,{#1}_2,\dots,{#1}_{#2}}
\def\ulist#1.#2.{{#1}^1,{#1}^2,\dots,{#1}^{#2}}
\def\lomitlist#1.#2.{{#1}_1,{#1}_2,\dots,\hat {{#1}_i}, \dots, {#1}_{#2}}
\def\lomitlistz#1.#2.{{#1}_0,{#1}_1,\dots,\hat {{#1}_i}, \dots, {#1}_{#2}}
\def\loc#1.#2.{\Cal O_{#1,#2}}
\def\fderiv#1.#2.{\frac {\partial #1}{\partial #2}}
\def\deriv#1.#2.{\frac {d #1}{d #2}}

\def\map#1.#2.{#1 \longrightarrow #2}
\def\rmap#1.#2.{#1 \dasharrow #2}
\def\emb#1.#2.{#1 \hookrightarrow #2}
\def\non#1.#2.{\text {Spec }#1[\epsilon]/(\epsilon)^{#2}}
\def\Hi#1.#2.{\text {Hilb}^{#1}(#2)}
\def\sym#1.#2.{\operatorname {Sym}^{#1}(#2)}
\def\Hb#1.#2.{\text {Hilb}_{#1}(#2)}
\def\Hm#1.#2.{\Hom_{#1}(#2)}
\def\prd#1.#2.{{#1}_1\cdot {#1}_2\cdots {#1}_{#2}}
\def\Bl #1.#2.{\operatorname {Bl}_{#1}#2}
\def\pl #1.#2.{#1^{\otimes #2}}
\def\mgn#1.#2.{\overline {M}_{#1,#2}}
\def\ialist#1.#2.{{#1}_1 #2 {#1}_2, #2\dots}
\def\pair#1.#2.{\langle #1, #2\rangle}
\def\vandermonde#1.#2.{\left|
\begin{matrix}
1 & 1 & 1 & \dots & 1\\
{#1}_1 & {#1}_2 & {#1}_3 & \dots & {#1}_{#2}\\
{#1}_1^2 & {#1}_2^2 & {#1}_3^2 & \dots & {#1}_{#2}^2\\
\vdots & \vdots & \vdots & \ddots & \vdots\\
{#1}_1^{#2-1} & {#1}_2^{#2-1} & {#1}_2^{#2-1} & \dots & {#1}_{#2}^{#2-1}\\
\end{matrix}
\right|
}
\def\vandermondet#1.#2.{\left|
\begin{matrix}
1 & {#1}_1   & {#1}_1^2 & \dots & {#1}_1^{#2-1}\\
1 & {#1}_2   & {#1}_2^2 & \dots & {#1}_2^{#2-1}\\
1 & {#1}_3   & {#1}_3^2 & \dots & {#1}_3^{#2-1}\\
\vdots & \vdots & \vdots & \ddots & \vdots\\
1 & {#1}_{#2}& {#1}_{#2}^2 & \dots & {#1}_{#2}^{#2-1}\\
\end{matrix}
\right|
}
\def\gr#1.#2.{\mathbb{G}(#1,#2)}

\def\alist#1.#2.#3.{{#1}_1 #2 {#1}_2 #2\dots #2 {#1}_{#3}}
\def\zlist#1.#2.#3.{#1_0 #2 #1_1 #2\dots #2 #1_{#3}}
\def\lomitlist30#1.#2.#3.{{#1}_0,{#1}_1 #2 \dots #2\hat {{#1}_i} #2\dots #2 {#1}_{#3}}
\def\lmap#1.#2.#3.{#1 \overset{#2}{\longrightarrow} #3}
\def\mes#1.#2.#3.{#1 \longrightarrow #2 \longrightarrow #3}
\def\ses#1.#2.#3.{0\longrightarrow #1 \longrightarrow #2 \longrightarrow #3 \longrightarrow 0}
\def\les#1.#2.#3.{0\longrightarrow #1 \longrightarrow #2 \longrightarrow #3}
\def\res#1.#2.#3.{#1 \longrightarrow #2 \longrightarrow #3\longrightarrow 0}
\def\Hi#1.#2.#3.{\text {Hilb}^{#1}_{#2}(#3)}
\def\ten#1.#2.#3.{#1\underset {#2}{\otimes} #3}
\def\lomitlist30#1.#2.#3.{{#1}_0 #2 {#1}_1 #2 \dots #2 \hat {{#1}_i} #2 \dots #2 {#1}_{#3}}
\def\mderiv#1.#2.#3.{\frac {d^{#3} #1}{d #2^{#3}}}

\def\Hom{\operatorname{Hom}}

\def\Supp{\operatorname{Supp}}

\def\Exc{\operatorname{Exc}}

\def\dim{\operatorname{dim}}

\def\deg{\operatorname{deg}}

\def\det{\operatorname{det}}

\def\Sing{\operatorname{Sing}}

\def\rest{\operatorname{res}}

\def\e{\Cal E}

\def\e1{E_1}
\def\e2{E_2}

\usepackage{hyperref,amscd,mathtools}
\newcommand{\nlc}[1]{\operatorname{nlc} \, #1}

\title{Positivity  of the moduli part}


\author{Florin Ambro}
\address{Institute of Mathematics Simion Stoilow of the Romanian Academy\\
 P.O. BOX 1-764\\
  RO014700 Bucharest, Romania}
  \email{florin.ambro@imar.ro}

\author{Paolo Cascini}
\address{Department of Mathematics\\
Imperial College London\\
180 Queen's Gate\\
London SW7 2AZ, UK}
\email{p.cascini@imperial.ac.uk}

\author{Vyacheslav Shokurov}
\address{Department of Mathematics, The Johns Hopkins University, Baltimore, MD 21218, USA}
\address{
Steklov Mathematics Institute, Russian Academy of Sciences, Cubkina Str. 8, 119991,
Moscow, Russia
}
\email{shokurov@math.jhu.edu}

\author{Calum Spicer}
\address{Department of Mathematics, King's College London, Strand,
London WC2R 2LS, UK}
\email{calum.spicer@kcl.ac.uk}

\begin{document}
\maketitle
\begin{abstract}We prove the Cone Theorem for algebraically integrable foliations. As a consequence, we show that termination of flips implies the $b$-nefness
of the moduli part of a log canonical pair with  respect to a contraction, generalising the case of lc trivial fibrations. 
\end{abstract}
\tableofcontents
\section{Introduction}

Starting from the work of Kodaira for elliptic surfaces, the canonical bundle formula has played an important role in birational geometry. Indeed, given any fibration between  normal complex varieties and with trivial relative canonical divisor,  this formula allows one to 
 compare the canonical divisor of  the domain with the canonical divisor  of the target. More specifically, 
let $(X,B)$ be a log canonical pair and let $f\colon X\to Z$ be a contraction between normal varieties over $\mathbb C$, i.e. a proper surjective morphism with connected fibres. 
If $f$ is lc trivial, i.e. if $K_X+B\sim_{\mathbb R} f^*D$ for some $\mathbb R$-Cartier $\mathbb R$-divisor $D$ on $Y$, then we may write $D=K_Z+B_Z+M_Z$, where $B_Z$ is the discriminant of the pair $(X/Z,B)$  and $M_Z$ is its moduli part. Roughly speaking, $B_Z$ measures the singularities of the fibres of $f$, whilst $M_Z$ measures how far  the fibration is from being isotrivial.
In particular, it is known that, after possibly replacing the fibration $f$ by taking a base change through a birational morphism $\alpha \colon Z'\to Z$,  the moduli part $M_Z$ of $(X/Z,B)$ is nef and for any further birational base change the moduli part of the induced pair coincides with the pull back of $M_Z$ (e.g. see \cite{Kawamata98,Ambro04,Kollar07,PS09}).

 Building on the work of the third author (cf. \cite{Shokurov13, Shokurov20}),
the aim of this paper is to generalise the above results to fibrations which are not necessarily lc trivial. 
Indeed, given a log canonical pair $(X,B)$ and a contraction $f\colon X\to Z$ between normal varieties, we can easily generalise the notion of discriminant and moduli part of the pair $(X/Z,B)$ (cf. Section \ref{s_moduli}). In contrast to the case of lc trivial fibrations,  
the moduli part of $(X/Z,B)$ is no longer a divisor on $Z$ but a divisor on $X$. On the other hand, as in the lc trivial case, 
if $B$ has rational coefficients, then  after possibly replacing the contraction $f\colon X\to Z$ by taking a base change through a birational morphism, 
it follows that the moduli part of $(X/Z,B)$ is compatible with pull-backs 
(cf. Theorem \ref{t_BPP}).
Moreover, if $K_X+B$ is relatively pseudo-effective over $Z$, and assuming termination of log canonical flips, we can run a Minimal Model Program $\phi\colon X\dashrightarrow Y$ of $K_X+B$ over $Z$ so that $K_Y+\phi_*B$ is relatively nef. Under suitable conditions, we prove that this implies the nefness of the moduli part of $(Y/Z,\phi_*B)$ (cf. Theorem \ref{thm_main}). As a consequence we have:

\begin{theorem}\label{t_main1}
Assume termination of log canonical flips in dimension $n$. 

Let $(X,B)$ be a  log canonical pair of dimension $n$ and let $f\colon X\to Z$ be a projective contraction  such that $K_X+B$ is relatively nef and $B\ge 0$. 

Then there exist a  log canonical pair $(Y,C)$ and a  commutative diagram
\[
\begin{tikzcd}
  Y \arrow[r, "\beta", dashrightarrow] \arrow[d, "f'"'] & X \arrow[d, "f"] \\
  Z' \arrow[r, "\alpha"] &  Z
 \end{tikzcd}
\]
where $\alpha\colon Z'\to Z$ is a projective birational morphism and $\beta\colon Y\dashrightarrow X$ is a birational map such that 
\begin{enumerate}
\item $(X,B)$ and $(Y,C)$ are crepant over the generic point of $Z$ (cf. Definition \ref{d_crepant});
\item  the moduli part $M_{Y}$ of $(Y/Z',C)$ is nef;
\item $(Y/Z',C)$ is BP stable and log stable (cf. Definition \ref{d_bpstable} and Definition \ref{d_logstable}); and
\item $(Y/Z',C)$ is maximal (cf. Definition \ref{d_maximal}). 
\end{enumerate}
\end{theorem}

We remark that in contrast to the lc trivial situation, it is not the case in general that the moduli
part $M_{X'}$ can be taken to be semi-ample (cf. Section \ref{s_counterexample_semi_ample}).

\medskip

The main ingredient of the proof of the above Theorem is the study of the birational geometry of algebraically integrable foliations, i.e. foliations induced by a dominant map $X\dashrightarrow Z$. Indeed, given a contraction $f\colon X\to Z$, by taking a weakly semi-stable reduction (cf. Theorem \ref{t_AK}) and by adding a 
vertical $\mathbb R$-divisor to $B$, we may assume that $f$ is toroidal and equidimensional and that the moduli part of $(X/Z,B)$ coincides with 
$K_{\cal F}+\Delta$ where $K_{\cal F}$ is the canonical divisor of the induced foliation 
$\cal F$ and $\Delta$ is the horizontal part of $B$ (cf. Proposition \ref{prop_*comparison}). 
Moreover, this property is preserved after running a relative Minimal Model Program over $Z$ (cf. Proposition \ref{prop_MMP_preserves_*}). 
It now remains to show that if the  moduli part of the induced pair is not nef, then there exists a rational curve which is vertical with respect to the fibration. 
To this end, we prove a stronger result: the Cone Theorem (cf. Theorem \ref{thm_cone}) holds for algebraically integrable foliations (see also \cite{bm16,Spicer20,CS21} for other versions of the Cone Theorem for foliations). 
This implies that if $(\cal F,\Delta)$ is a foliated log pair (cf. Section \ref{s_lcf}) 
then any $(K_{\cal F}+\Delta)$-negative extremal ray is spanned by a rational curve which is contained in a fibre of the induced morphism. 

Elsewhere in the literature results like Theorem \ref{t_main1} are proven by analysing the semi-positivity properties
of $f_*\omega_{X/Y}$, which in turn is based on the study of variations of Hodge structures.  By using foliations, 
our proofs avoid such techniques.

\medskip

We also provide several other applications of Theorem \ref{thm_cone}. 
For instance, we prove a version of Theorem \ref{t_main1} where ``nef" is replaced by  ``pseudo-effective",
(cf.  Theorem \ref{thm_relatively_pseudo_eff}).
An important consequence of the main results of \cite{CP19} is that if a projective manifold admits a foliation $\cal F$
where $K_{\cal F}$ is not pseudo-effective, then $K_X$ is not pseudo-effective \cite[Theorem 1.2]{CP19}.
We  prove the existence of the $K_{\cal F}$-MMP under some natural and mild conditions which allows 
us to prove a version of \cite[Theorem 1.2]{CP19} on singular varieties.

\begin{theorem}
\label{t_main3}
Let $X$ be a klt projective 
variety and let $(\cal F,\Delta)$ be a foliated pair such that $K_{\cal F}+\Delta$ is not pseudo-effective
and such that $(X, \Delta)$ is log canonical.

Then $K_X+\Delta$ is not pseudo-effective.  Moreover, $X$ is covered by $(K_X+\Delta)$-negative curves
tangent to $\mathcal F$.
\end{theorem}

This is a direct consequence of Theorem \ref{thm_KF_KX} where we prove a 
slightly more general statement.

\medskip 

Finally, the following theorem is a slight generalisation of a theorem by Ambro
\cite[Theorem 0.2]{Ambro04} in the klt case and  Koll\'ar \cite[Proposition 8.4.9]{Kollar07} in the log canonical case. 
Note  that in contrast to the existing proofs, our proof does not make use of variation of Hodge structures, but it relies mainly on methods from the Minimal Model Program (see Section \ref{s_bnefness}). 

\begin{theorem} \label{t_f-trivial}
Let $X$ be a quasi projective normal variety of dimension $n$ and let 
$(X/Z,B)$ be a GLC pair  such that    $B\ge 0$ and such that the induced morphism $f\colon X\to Z$ is a projective contraction.  Assume that 
\begin{enumerate}
\item $(X/Z,B)$ is BP stable over $Z$ (cf. Definition \ref{d_bpstable}); and 
\item $K_X+B\sim_{f,\mathbb R}0$.
\end{enumerate}

Then the moduli part $M_X$ of $(X/Z,B)$ is nef. 
\end{theorem}

\subsection{ Acknowledgements.} The authors would like to thank J. M\textsuperscript cKernan for many useful discussions.

\section{Preliminaries}

We work over an algebraically closed field of characteristic zero. Let $X$ be a normal variety and let $D$ be a $\mathbb R$-divisor on $X$. For any prime divisor $S$ in $X$, we define $m_SD$ to be the coefficient of $D$ along $S$. 
If $f\colon X\dashrightarrow Z$ is a dominant map between normal varieties, we say that a subvariety $W\subset X$ is {\bf horizontal} with respect to $f$ if $W$ dominates $Z$, otherwise we say that $W$ is {\bf vertical} with respect to $f$. 
A {\bf contraction} $f\colon X\to Z$ is a surjective proper morphism between normal varieties such that $f_*\mathcal O_X=\mathcal O_Z$.  Given a reduced divisor $P$ on $Z$, we denote by $f^{-1}(P)$ its preimage as a reduced divisor on $X$. 

Two surjective morphisms $f\colon X\to Z$ and $f'\colon X'\to Z'$ between normal varieties are said to be {\bf birationally equivalent} if there exists a commutative diagram
\[
\begin{tikzcd}
  X' \arrow[r, "\beta",dashrightarrow] \arrow[d, "f'"'] & X \arrow[d, "f"] \\
  Z' \arrow[r, "\alpha",dashrightarrow] &  Z
 \end{tikzcd}
\]
where $\alpha$ and $\beta$ are birational maps. 

A log pair $(X,B)$ consists of a normal variety $X$ and a $\mathbb R$-divisor $B$ on $X$ such that $K_X+B$ is $\mathbb R$-Cartier. We refer to \cite[Section 2.3]{KM98} for the classical definitions of singularities 
(e.g. klt, log canonical, etc.. ) appearing in the Minimal Model Program.

\subsection{Toroidal pairs}

\begin{definition}
Let $(X,\Sigma_X)$ be a log pair, where $\Sigma_X\ge 0$ is a reduced divisor.  We say that $(X,\Sigma_X)$ is a {\bf toroidal pair} if for every closed point $x\in X$, there exists a toric variety $X_\sigma$, a  closed point $t\in X_\sigma$ and an isomorphism of complete local algebras 
\[
\phi_x\colon \widehat{\mathcal O}_{X,x}\simeq \widehat{\mathcal  O}_{X_\sigma,t}
\]
such that the ideal of $\Sigma_X$ maps to the invariant ideal of $X_{\sigma}\setminus T_{\sigma}$ where $T_\sigma\subset X_\sigma$ is the maximal torus. The pair $(X_\sigma, t)$ is called a {\bf local model} at $x\in X$.

Let $(X,\Sigma_X)$ and $(Z,\Sigma_Z)$ be toroidal log pairs. A dominant morphism $f\colon X\to Z$ is a  {\bf toroidal morphism} if for every closed point $x\in X$ there exist local models $(X_\sigma,t)$ at $x$ and $(Z_{\tau},s)$ at $f(x)$, and a toric morphism $g\colon X_\sigma\to Z_{\tau}$ so that the following diagram commutes
\[
\begin{tikzcd}
\widehat {\mathcal O}_{X,x} \arrow[r]& \widehat{\mathcal O}_{X_{\sigma},t} \\
\widehat {\mathcal O}_{Z,f(z)}\arrow[r]\arrow[u] &\widehat {\mathcal O}_{Z_{\tau},s}\arrow[u]
 \end{tikzcd}
\]
where the vertical maps are the algebra homomorphisms induced by $f$ and $g$ respectively. 
\end{definition}

\begin{theorem}\label{t_AK}
Let $(X,B)$ be a log pair and let $f\colon X\to Z$ be a contraction between normal varieties.

Then there exist toroidal pairs   $(X',\Sigma_{X'})$ and $(Z',\Sigma_{Z'})$ and a diagram 

\[
\begin{tikzcd}
  X' \arrow[r, "\beta"] \arrow[d, "f'"'] & X \arrow[d, "f"] \\
  Z' \arrow[r, "\alpha"] &  Z
 \end{tikzcd}
\]

such that 
\begin{enumerate}
\item $\alpha$ and $\beta$ are birational projective morphisms,  
\item $f'$ is a toroidal morphism;
\item $(Z',\Sigma_{Z'})$ is log smooth; 
\item the support of $\beta^{-1}_*B+\Exc \beta$ is contained in $\Sigma_{X'}$; and
\item $X'$ admits quotient toric singularities and $f'$ is equidimensional. 
%
\end{enumerate}
In particular, $X'$ is $\mathbb Q$-factorial. 
\end{theorem}
\begin{proof}
This follows from \cite[Theorem 2.1 and Proposition 4.4]{AK00}.
\end{proof}

\subsection{Discriminant and Moduli Part}\label{s_moduli}
Throughout this paper, a pair $(X/Z,B)$ consists of a  surjective proper morphism  $f\colon X\to Z$ between normal varieties and a log pair $(X,B)$. We say that the pair $(X/Z,B)$ is {\bf generically log canonical}, in short GLC, 
if $Z$ is irreducible and 
the pair $(X,B)$ is log canonical over the generic point of $Z$.

Let $(X/Z,B)$ be a GLC pair. Consider a commutative diagram 
\[
\begin{tikzcd}
  X' \arrow[r, "\beta"] \arrow[d, "f'"'] & X \arrow[d, "f"] \\
  Z' \arrow[r, "\alpha"] &  Z
 \end{tikzcd}
\]
where $f'$ is a surjective proper morphism between normal varieties, $Z'$ is irreducible and $\beta$ and $\alpha$ are generically finite surjective projective morphisms.
We may write
 \[
K_{X'}+B'=\beta^*(K_X+B).
\]
Then the {\bf induced pair} $(X'/Z',B')$ is GLC (cf. \cite[Proposition 5.20]{KM98}).

\begin{definition}\label{d_crepant}
Let $(X/Z,B)$ and $(X'/Z',B')$ be  GLC pairs such that the induced morphisms $f\colon X\to Z$ and $f'\colon X'\to Z'$ are birational equivalent contractions. We say that $(X/Z,B)$ and $(X'/Z',B')$ are {\bf crepant} over the generic point of $Z$ if, given birational morphisms  $p\colon W\to X$ and $q\colon W\to X'$ which resolve the indeterminacy of the induced map $X\dashrightarrow X'$, we have that 
\[p^*(K_X+B)-q^*(K_{X'}+B')\] is vertical with respect to the induced morphism $W\to Z$. 
\end{definition}

\medskip
Let $(X/Z,B)$ be a GLC pair. For any prime divisor $P\subset Y$, we define
\[
\gamma_P=\sup\{t\in \mathbb R\mid (X,B+tf^*P) \text{ is log canonical over the generic point of }P\}.
\]
The {\bf discriminant } of $(X/Z,B)$ is the $\mathbb R$-divisor 
\[
B_Z:=\sum_P (1-\gamma_P)P.
\]

We now define the moduli part of a GLC pair $(X/Z,B)$ such that  the induced morphism $f\colon X\to Z$ is a contraction. 
We first assume  that $f$ is equidimensional and 
$K_Z+B_Z$ is $\mathbb R$-Cartier. In this case, we define the moduli part $M_X$ of $(X/Z,B)$ as 
\[
M_X\coloneqq K_X+B-f^*(K_Z+B_Z). 
\]
Note that $M_X$ is only defined up to linear equivalence, as it depends on the choice of the canonical divisors $K_X$ and $K_Z$. 

Consider now any GLC pair $(X/Z,B)$, induced by a morphism $f\colon X\to Z$. By Theorem \ref{t_AK}, there exists an equidimensional contraction $f'\colon X'\to Z'$ which is birational equivalent to $f$, such that $Z'$ is smooth and the induced maps $\alpha\colon Z'\to Z$ and $\beta\colon X'\to X$ are birational morphisms. Let $M_{X'}$ be the moduli part of the induced pair $(X'/Z',B')$. We then define the {\bf moduli part} of $(X/Z,B)$  as
\[M_X\coloneqq\alpha_*M_{X'}.
\]
As above, $M_X$ is only defined up to linear equivalence. It is easy to check that $M_X$ does not depend on the choice of $f'$.

\begin{lemma}
\label{lem_base_change} 
Let $(X/Z,B)$ be a GLC pair and let $\alpha\colon Z'\to Z$ be a finite morphism.
 Consider the following commutative diagram 
\[
\begin{tikzcd}
  X' \arrow[r, "\beta"] \arrow[d, "f'"'] & X \arrow[d, "f"] \\
  Z' \arrow[r, "\alpha"] &  Z
 \end{tikzcd}
\]
and let $(X'/Z',B')$ be the induced GLC pair. 

Then 
\[
K_{Z'}+B'_{Z'} = \alpha^*(K_Z+B_Z).
\]

\end{lemma}

\begin{proof}
See \cite[Theorem 3.2]{Ambro98}.
\end{proof}

\subsection{Boundary property}
\begin{definition}\label{d_bpstable}
Let  $(X/Z,B)$ be a GLC pair induced by a morphism $f\colon X\to Z$. The pair $(X/Z,B)$ is said to be {\bf BP stable} (resp. {\bf BP semi-stable}) over $Z$ if $K_Z+B_Z$ is $\mathbb R$-Cartier and  
for any contraction $f'\colon X'\to Z'$ which is birational equivalent to $f$ and such that the induced maps $\alpha\colon Z'\to Z$ and $\beta\colon X'\to X$ are projective birational morphisms, if 
 $B'_{Z'}$ is the discriminant of the induced pair $(X'/Z',B')$ and if we write
\[
K_{Z'}+B'_{Z} = \alpha^*(K_Z+B_Z)
\]
then $B'_Z=B'_{Z'}$ (resp. $B'_Z\ge B'_{Z'}$). 

We say that a GLC pair $(X/Z,B)$ induced by a morphism $f\colon X\to Z$ satisfies the {\bf boundary property} if there exist  a surjective projective morphism $f'\colon X'\to Z'$ between normal varieties, projective birational  morphisms $\alpha\colon Z'\to Z$  and $\beta\colon X'\to X$, and a commutative diagram 
\[
\begin{tikzcd}
  X' \arrow[r, "\beta"] \arrow[d, "f'"'] & X \arrow[d, "f"] \\
  Z' \arrow[r, "\alpha"] &  Z
 \end{tikzcd}
\]
such that the induced pair $(X'/Z',B')$ is BP stable over $Z'$. 
\end{definition}

\begin{remark}\label{r_induced}
\ 

\begin{enumerate}
\item Let $(X/Z,B)$ be a GLC pair and let $\beta\colon X'\to X$ be a birational morphism between normal varieties. Let $(X'/Z,B')$ be the induced GLC pair. 
Then the discriminant $B_Z$ of $(X/Z,B)$ coincides with the discriminant $B'_Z$ of $(X'/Z,B')$. In particular,  $(X/Z,B)$ is BP stable over $Z$ if and only if $(X'/Z,B')$ is BP stable over $Z$. 
\item Let $(X/Z,B)$ be a GLC pair induced by a contraction $f\colon X\to Z$ and which is BP stable over $Z$. Then, it is easy to check that 
\[
M_X= K_X+B-f^*(K_Z+B_Z). 
\]
\end{enumerate}
\end{remark}

\begin{example} Let $(X/Z,B)$ be a GLC pair induced by a birational morphism $f\colon X\to Z$. Then $B_Z=f_*B$ and $M_X=0$. It follows that  $(X/Z,B)$ is BP stable over $Z$ if and only if $K_Z+B_Z$ is $\mathbb R$-Cartier and $f$ is $(K_X+B)$-crepant, i.e. 
\[
K_X+B=f^*(K_Z+B_Z).
\]
\end{example}
\medskip 

Although it will not be used in the rest of the paper, we recall the following

\begin{theorem}\cite[Theorem 1]{Shokurov20}\label{t_BPP}
Let $(X/Z,B)$ be a GLC pair 
and assume that $B$ is a $\mathbb Q$-divisor. 

Then $(X/Z,B)$ satisfies the boundary  property.
\end{theorem}

\subsection{Log stable pairs}

\begin{lemma}\label{l_valuation}
Let $X$ be a normal variety and let $F$ be a valuation over $X$.

Then there exists a $\mathbb Q$-divisor $B$ such that $(X,B)$ is log canonical and $F$ is the only log canonical centre of $(X,B)$. 
\end{lemma}
\begin{proof}
Let $\Delta$ be a $\mathbb Q$-divisor such that $(X,\Delta)$ is a log pair. We may assume that 
$\rfdown \Delta.\le 0$. Let $\pi\colon Y\to X$ be a log resolution of $(X,\Delta)$ such that $F$ is a divisor on $Y$. We may write
\[
K_Y+\Gamma=f^*(K_X+\Delta). 
\]
Let $t\in \mathbb Q$ be such that $m_F(\Gamma+tF)=1$ and 
let $A_1,A_2\ge 0$ sufficiently general ample $\mathbb Q$-divisors on $Y$ such that 
\[\Gamma+tF+A_1-A_2\sim_{\mathbb Q}0,\]
 $(Y,\Gamma+tF+A_1-A_2)$ is log canonical and $F$ is its only log canonical place.
 We may then define $B\coloneqq\Delta+\pi_*(A_1-A_2)$, to obtain the desired $\mathbb Q$-divisor. 
\end{proof}

\begin{definition}\label{d_logstable}
Let $(X/Z,B)$  be a GLC pair induced by a contraction $f\colon X\to Z$ and let $B_Z$ be its discriminant. Then $(X/Z,B)$ is said to be {\bf log stable} if for any $\mathbb R$-Cartier divisor $H$ on $Z$ we have that 
$(X,B+f^*H)$ is log canonical if and only if $(Z,B_Z+H)$ is log canonical. 
\end{definition}

\begin{theorem}
\cite[Proposition 4]{Shokurov20} 
\label{t_logstable}
Let $(X/Z,B)$ be a GLC pair induced by a contraction $f\colon X\to Z$.

Then $(X/Z,B)$ is log stable if and only if $(X/Z,B)$ is BP stable over $Z$. 
\end{theorem}

\begin{proof}
Let $B_Z$ be the discriminant of $(X/Z,B)$. 

Assume first that $(X/Z,B)$ is BP stable over $Z$. Let $H$ be a $\mathbb R$-Cartier divisor on $Z$ and note that $B_Z+H$ is the discriminant of $(X/Z,B+f^*H)$. Assume that $(X,B+f^*H)$ is not log canonical. Then, by Remark \ref{r_induced}, we may assume that $B+f^*H$ has a component $S$ of coefficient greater than one. Since $(X/Z,B+f^*H)$ is BP stable over $Z$, after possibly replacing $X$ and $Z$ by higher models, we may assume that $f(S)$ is a divisor on $Z$ and the coefficient of $B_Z+H$ along such a divisor is greater than one.  Thus, $(Z,B_Z+H)$ is not log canonical. Vice versa, assume that $(Z,B_Z+H)$ is not log canonical. Then, since $(X/Z,B+f^*H)$ is BP stable over $Z$, after possibly replacing $X$ and $Z$ by higher models, we may assume that a component of $B_Z+H$ has coefficient greater than one and, therefore, $(X,B+f^*H)$ is not log canonical. Thus, $(X/Z,B)$ is log stable. 

Assume now that $(X/Z,B)$ is log stable. Let $F$ be a valuation over $Z$. By Lemma \ref{l_valuation}, there exists a $\mathbb R$-Cartier divisor $H$ on $Z$ such that $(Z,B_Z+H)$ is log canonical and $F$ is the unique log canonical place for $(Z,B_Z+H)$. 
Let $\alpha \colon Z'\to Z$ be a birational morphism from a normal variety $Z'$ such that $F$ is a divisor on $Z'$ and let $K_{Z'}+B'_{Z}=\alpha^*(K_X+B_Z)$.
Since $(X/Z,B)$ is log stable, it follows that $(X,B+f^*H)$ is log canonical.
By Theorem \ref{t_AK} and Remark \ref{r_induced}, after possibly replacing $X$ and $Z'$ by  higher models, we may assume that there exist reduced divisors $\Sigma_X$ and $\Sigma_{Z'}$ on $X$ and $Z'$ respectively such that $(X,\Sigma_X)$ and $(Z',\Sigma_{Z'})$ are toroidal, the support of $B+f^*H$ is contained in $\Sigma_X$ and that the induced map $f'\colon X\dashrightarrow Z'$ is a morphism which is an 
equidimensional toroidal contraction.
Let $B_{Z'}$ is the discriminant of $(X/Z',B)$ and suppose by contradiction that $m_{F}(B_{Z'}+f^*H)<1$. Then, for any divisor $R$ on $X$ dominating $F$, we have that $m_R(B+f^*H)<1$ and, in particular, 
$R$ does not contain any log canoncal centre of $(X,B+f^*H)$.
Thus, for any sufficiently small $\mathbb R$-Cartier $\mathbb R$-divisor $\Theta\ge 0$ on $Z$, we have that $(X,B+f^*(H+\Theta))$ is log canonical along $R$, but if the support of $\Theta$ is a sufficiently general $\mathbb R$-divisor  which contains the centre of $F$ in $Z$ then $(Z,B_Z+f^*(H+\Theta))$ is not log canonical but $(X,B+f^*(H+\Theta))$ is, contradicting the fact that $(X/Z,B)$ is log stable. 

Thus, 
\[m_{F}(B_{Z'}+f^*H)=1=m_F(B'_Z+f^*H)
\]
and, therefore, $(X/Z,B)$ is BP stable over $Z$, as claimed. 								
\end{proof}

\subsection{Property $(*)$}

We begin with the following
\begin{lemma}\label{l_fibre}
Let $(X,B)$ be a log pair such that $B\ge 0$ and let $f\colon X\to Z$ be a contraction between normal varieties of dimension $n$ and $m$ respectively. Let $z\in Z$ be a closed point and assume that $D_1,\dots,D_m$  are effective Cartier divisors on $Z$ containing $z$ such that $(X,B+\sum_{i=1}^m f^*D_i)$ is log canonical around $f^{-1}(z)$. 

Then any irreducible component of $f^{-1}(z)$ has dimension $n-m$. 
\end{lemma}

\begin{proof}
Suppose not. Then there exists a component $G$ of $f^{-1}(z)$ of codimension $k$ where $k<m$. In particular, it follows that $m>1$ and $k\ge 1$. 
We proceed by induction on $k$. If $k=1$, then $G$ is a divisor and  the coefficient of $\sum_{i=1}^m f^*D_i$ along $G$ is at least $m$, a contradiction. 

Assume now that $k>1$. Let $X'$ be the normalisation of a component of $f^*D_k$ containing $G$
and let $Z'$ be the normalisation of $f(X')$. By adjunction and by induction, it follows that $G$ has dimension equal to $\dim X'-\dim Z'=n-m$, a contradiction. Thus, our claim follows. 
\end{proof}

\begin{definition}\label{d_property*}
Let $(X/Z,B)$ be a GLC pair, defined by a  morphism $f\colon X\to Z$.  
We say that $(X/Z,B)$ satisfies {\bf Property $(*)$} if $f$ is a projective  contraction and the following hold:
\begin{enumerate}
\item  there exists a reduced divisor $\Sigma_Z$ on $Z$  such that $(Z,\Sigma_Z)$ is log smooth and 
the vertical part of $B$ coincides with 
$f^{-1}(\Sigma_Z)$; and 
\item for any closed point $z\in Z$ and for any  $\Sigma\ge \Sigma_Z$ reduced divisor on $Z$ such that $(Z,\Sigma)$ is log smooth around $z$,  we have that $(X,B+f^*(\Sigma-\Sigma_Z))$ is log canonical around $f^{-1}(z)$.
\end{enumerate}
\end{definition}

\begin{lemma}\label{l_star}
Let $(X/Z,B)$ be a GLC pair satisfying Property $(*)$. Let $\Sigma_Z$ be the induced divisor on $Z$ and let $f\colon X\to Z$ be the induced contraction. 

Then 
\begin{enumerate}
\item $(X,B)$ is log canonical; 
\item  the discriminant  $B_Z$ of $(X/Z,B)$ coincides with $\Sigma_Z$; and
\item  if $B\ge 0$ then $f$ is equidimensional outside $\Sigma_Z$. 
\end{enumerate}
\end{lemma}

\begin{proof}
 (1) follows immediately by the definition of Property $(*)$. In particular, the coefficients of $B_Z$ are at most one.

We now prove (2). Since the vertical part of $B$ coincides with $f^{-1}(\Sigma_Z)$, it follows that $\Sigma_Z\le B_Z$. 
 Let $P$ be a prime divisor which is not contained in $\Sigma_Z$ and let $\eta\in P$ be its general point. Then $(Z,\Sigma_Z+P)$ is log smooth at $\eta$ and, therefore, $(X,B+f^*P)$ 
is log canonical along $f^{-1}(\eta)$, which implies that $P$ is not contained in the support of $B_Z$. Thus, (2) follows.

We finally prove (3). Let $z\in Z$ be a closed point which is not contained in $\Sigma_Z$ and, if $m$ denotes the dimension of $Z$, let $\Sigma_1,\dots,\Sigma_m$ be distinct hypersurfaces in $Z$ such that $z\in \Sigma_i$ for $i=1,\dots,m$ and $(Z,\Sigma_Z+\sum_{i=1}^m \Sigma_i)$ is log smooth at $z$. Then $(X,B+\sum_{i=1}^m f^*\Sigma_i)$ is log canonical along $f^{-1}(z)$ and Lemma \ref{l_fibre} implies that the dimension of every component of $f^{-1}(z)$ is equal to $\dim X-\dim Z$. Thus, (3) follows. 
\end{proof}

\begin{lemma}\label{l_lcc} Let $(X/Z,B)$ be a GLC pair satisfying Property $(*)$ such that $B\ge 0$. Let $\Sigma_Z$ be the induced divisor on $Z$ and let $f\colon X\to Z$ be the induced contraction. 
Let $\Sigma\ge \Sigma_Z$ be a reduced divisor on $Z$ such that $(Z,\Sigma)$ is log smooth and let $P$ be a stratum of $\Sigma$. 

Then any irreducible component of $f^{-1}(P)$ is a log canonical centre of $(X,B+f^*(\Sigma-\Sigma_Z))$.
\end{lemma}

\begin{proof} Let $k$ be the codimension of $P$ in $Z$. By assumption, there exist irreducible components $\Sigma_1,\dots, \Sigma_k$ of $\Sigma$ such that $P=\bigcap_{i=1}^k \Sigma_i$. Since the vertical part of $B$ coincides with $f^{-1}(\Sigma_Z)$ and since $(X,B+f^*(\Sigma-\Sigma_Z))$ is log canonical along $f^{-1}(\Sigma_i)$, it follows that $f^{-1}(\Sigma_i)$ is a log canonical centre of $(X,B+f^*(\Sigma-\Sigma_Z))$.
Thus, any irreducible component of $f^{-1}(P)$ is an intersection of log canonical centres of $(X,B+f^*(\Sigma-\Sigma_Z))$ and the claim follows from \cite[Theorem 1.1]{Ambro11}. 
\end{proof}


%
%

\begin{proposition}\label{p_toroidal}
Let $(X,\Sigma_X)$ and $(Z,\Sigma_Z)$ be toroidal pairs such that $(Z,\Sigma_Z)$ is log smooth and let $f\colon X\to Z$ be a toroidal contraction. Let $(X,B)$ be a log canonical pair such that  $\Supp B\subset \Sigma_X$ and such that the vertical part of $B$ coincides with  $f^{-1}(\Sigma_Z)$

Then $(X/Z,B)$ satisfies Property $(*)$. 
\end{proposition}

\begin{proof} By assumption, $(X, B)$ satisfies (1) of Definition \ref{d_property*}.

To show that (2) of Definition \ref{d_property*} holds, it suffices to consider the case where $B = \Sigma_X$. 
Fix $z \in Z$ and $x \in f^{-1}(z)$.  Verifying (2) is local on $Z$ and $X$ and therefore
it is enough to show that (2) holds in an analytic neighbourhood of $x \in X$.

Let $\Sigma \geq \Sigma_Z$ be a reduced divisor such that
$(Z, \Sigma)$ is log smooth.
We may find a local model $(X_{\tau},t)$ at $z \in Z$ such that $\Sigma$ maps to 
a torus invariant divisor, which we will continue to denote by $\Sigma$.
By \cite[Proposition 3.2]{AK00}, we may assume that there exists  a local model $(X_{\sigma},s)$ of $x \in X$
so that the induced morphism between the local models $g\colon X_\sigma \rightarrow X_\tau$ is toric.
In particular, $g^{-1}(\Sigma)$ is contained in the torus
invariant boundary of $X_\sigma$.

Since $f\colon X \to Z$ is toroidal, if $D \subset Z$ is any divisor such that $f^*D$ is not reduced then $D \subset \Supp \Sigma_Z$.
It follows that $g^*(\Sigma-\Sigma_Z)$ is reduced.
Next, observe that $g^*(\Sigma -\Sigma_Z)$ has no components in common with $\Sigma_X$ and so
$g^*(\Sigma-\Sigma_Z)+\Sigma_X$ is a reduced divisor 
contained in the torus invariant boundary of $X_{\sigma}$.  It follows that 
$(X_\sigma, g^*(\Sigma-\Sigma_Z)+\Sigma_X)$ is log canonical, and hence $(X, f^*(\Sigma-\Sigma_Z)+\Sigma_X)$ is log canonical
in a neighborhood of $x \in X$, as required.
\end{proof}


%
%
%
%
%

\begin{proposition}
\label{prop_*resolution}
Let $(X/Z,B)$ be a GLC pair such that the induced morphism $f\colon X\to Z$ is a projective contraction. 

Then there exist a commutative diagram 
\[
\begin{tikzcd}
  X' \arrow[r, "\beta"] \arrow[d, "f'"'] & X \arrow[d, "f"] \\
  Z' \arrow[r, "\alpha"] &  Z
 \end{tikzcd}
\]
where $\alpha$ and $\beta$ are birational projective morphisms and a $\mathbb R$-divisor $B'$ on $X'$ such that 
\begin{enumerate}
\item $X'$ is $\mathbb Q$-factorial and $f'$ is an equidimensional contraction;
\item $(X'/Z',B')$ satisfies Property $(*)$; and
\item $K_{X'}+B' = \beta^*(K_X+B)+F$, where $F\ge 0$ is vertical with respect to $f'$. 
\end{enumerate}
\end{proposition}

Note that, using the same notation as Proposition \ref{prop_*resolution}, it follows that  $(X,B)$ and $(X',B')$ are crepant over the generic point of $Z$.

\begin{proof} 
The claim follows easily from Theorem \ref{t_AK} and Proposition \ref{p_toroidal}.
\end{proof}


%

\begin{proposition}
\label{prop_MMP_preserves_*}
Let $(X/Z,B)$ be a GLC pair satisfying Property $(*)$  and such that $B\ge0$.   Let $\phi\colon X\dashrightarrow Y$ be a sequence of steps of the $(K_X+B)$-MMP over $Z$. Let $C=\phi_*B$ and let $f\colon X\to Z$ and $g\colon Y\to Z$ be the induced contractions. 

Then $(Y/Z,C)$ satisfies Property $(*)$ and for any closed point $z\in Z$, the map $\phi^{-1}$ is an isomorphism along the generic point of any irreducible component of $g^{-1}(z)$. In particular, if $f$ is equidimensional, then $g$ is also equidimensional. 
\end{proposition}

\begin{proof}
We first prove that $(Y/Z,C)$ satisfies Property $(*)$. Let $\Sigma_Z$ be the induced divisor. 
Since the vertical part of $B$ coincides with $f^{-1}(\Sigma_Z )$  and $C=\phi_*B$, it follows that
 the vertical part of $C$ coincides with  $g^{-1}(\Sigma_Z)$.
 
 Let $z\in Z$ be a closed point and let $\Sigma\ge \Sigma_Z$ be a reduced divisor such that $(Z,\Sigma)$ is log smooth around $z$. Then $(X,B+f^*(\Sigma-\Sigma_Z))$ is log canonical around $f^{-1}(z)$ and since $\phi\colon X\dashrightarrow Y$ is also a sequence of steps for the $(K_X+B+f^*(\Sigma-\Sigma_Z))$-MMP, it follows that $(Y,C+g^*(\Sigma-\Sigma_Z))$ is log canonical around $g^{-1}(z)$.  Thus, $(Y/Z,C)$ satisfies Property $(*)$.

Now, let $z\in Z$ be a closed point and let  $G$ be an irreducible component of $g^{-1}(z)$. Let $\Sigma\ge \Sigma_Z$ be a reduced divisor such that $(Z,\Sigma)$ is log smooth and $z$ is a stratum of $\Sigma$.  Then, Lemma \ref{l_lcc} implies that $G$ is a log canonical centre of $(Y,C+g^*(\Sigma-\Sigma_Z))$. On the other hand, Property $(*)$ implies that $(X,B+f^*(\Sigma-\Sigma_Z))$ is log canonical around $f^{-1}(z)$. Thus, the negativity lemma (cf.  \cite[Lemma 3.38]{KM98}) implies that $\phi^{-1}$ is an isomorphism along the generic point of $G$, as claimed. 
\end{proof}

Note that if $(X/Z, B)$ is a GLC pair satisfying Property $(*)$, then in general it is not the case
that $(X/Z, B+D)$ satisfies Property $(*)$ for some ample $\mathbb Q$-divisor $D\ge 0$, as the following example shows.
Consider the projection $\pi\colon \mathbb P^2\times \mathbb P^1 \to \mathbb P^2$, let $Z=\mathbb P^2$ 
and let $\ell\subset Z$ be a line.  
Then $\pi^{-1}(\ell)$ is naturally isomorphic to $\mathbb P^1\times \mathbb P^1$. Let $X$ be the blow up of $\mathbb P^3$ along the diagonal of $\pi^{-1}(\ell)$ and let  $f\colon X \rightarrow Z$ be the induced morphism. Let $B = f^{-1}(\ell)$ and let $S \subset X$ denote the singular locus of $B$.  Then, it is easy to check that $(X/Z,B)$ is a GLC pair satisfying Property $(*)$, but 
 $(X/Z, B+D)$ fails to satisfy Property $(*)$ for any ample $\mathbb Q$-divisor $D\ge 0$. Indeed, any horizontal component of $\Supp D$  intersects $S$ at some closed point $x\in S$ and condition (2) of Definition \ref{d_property*} does not hold for $z=f(x)$. 
 
However, we do have a weak form of a Bertini-type result:

\begin{proposition}
\label{p_prop*_bertini}
Let $(X/Z,B)$ be a GLC pair satisfying Property $(*)$,  such that $B\ge0$ and the induced morphism $f\colon X\to Z$ is a projective contraction.
Let $A$ be an ample $\mathbb Q$-divisor and let $z \in Z$ be a closed point.

Then $(X/Z, B+H)$ satisfies Property $(*)$ in a neighbourhood of $f^{-1}(z)$
for a general choice of $0 \leq H \sim_{\mathbb Q} A$. 
\end{proposition}
\begin{proof}
Note that if $H$ is horizontal with respect to $f$ then $(X/Z, B+H)$ satisfies condition (1)
 of Definition \ref{d_property*}.  We will now show that  we can arrange for condition (2) to hold  in a neighbourhood of 
$f^{-1}(z)$.

We fix a reduced divisor $P\ge 0$ on $Z$, not containing any component of $\Sigma_Z$
such that $(Z, \Sigma_Z+P)$ is log smooth at $z$ and $z$ is a log canonical centre of $(Z,\Sigma_Z+P)$.   By assumption, $(X, B+f^*P)$ is log canonical around $f^{-1}(z)$. 
Let  $0\le H\sim_{\mathbb Q}A$ be sufficiently general so that $(X, B+H+f^*P)$ is log canonical around $f^{-1}(z)$.
We claim that if 
$P'\ge 0$ is any other reduced divisor, not containing any component of $\Sigma_Z$
and such that $(Z, \Sigma_Z+P')$ is log smooth at $z$ then $(X, B+H+ f^*P')$ is log canonical around $f^{-1}(z)$. The claim immediately implies that 
$(X/Z, B+H)$ satisfies Property $(*)$ around $f^{-1}(z)$. Thus, our result follows.

We now prove the claim by proceeding by induction on the number $k$ of components of $P$ which are not components of $P'$. If $k=0$ then there is nothing to prove. Thus, we may assume that $k\ge 1$. Let $P_1,\dots,P_q$ be the components of $P$. Then, after possibly enlarging $P'$, 
we may assume that  there exists a component $P'_1$ of $P'$ which is not contained in $P$ and such that if
\[\overline {P}=P'_1+\sum_{i=2}^qP_i,
\]
then $(Z,\Sigma_Z+\overline P)$ is log smooth at $z$. 
Let $m=\dim Z$ and let 
\[
Z_1\subset Z_2\subset \dots \subset Z_m= Z
\]
be a sequence of log canonical centres  of $(Z, \Sigma_Z+\sum_{i = 2}^{q}P_i)$ containing $z$ with $\dim Z_\ell = \ell$, 
for $\ell=1,\dots,m$. Let $\Gamma\coloneqq B+H+f^*(\sum_{i=2}^q P_i)$.  After possibly shrinking $Z$, we may assume that $P_1\vert_{Z_1} = P'_1\vert_{Z_1}$.

Note that, for any $\ell=1,\dots,m-1$, we have that $f^{-1}(Z_\ell)$ is an intersection of  log canonical centres of $(X,\Gamma)$ of codimension one and, by 
 \cite[Theorem 1.1]{Ambro11},
 it is therefore a log canonical centre. Let $X_\ell$ denote its normalisation. We have morphisms $\iota_\ell\colon X_\ell \rightarrow X_{\ell+1}$
which commute with the  inclusion $f^{-1}(Z_\ell)\hookrightarrow f^{-1}(Z_{\ell+1})$. Let 
$\Gamma_m=\Gamma$. By repeatedly applying adjunction (e.g. see \cite[Proposition 4.9]{Kollar13}), it follows that for each $\ell=1,\dots,m-1$, there exist a $\mathbb Q$-divisor $\Gamma_\ell$ on $X_\ell$, such that  $\iota_\ell(X_\ell)$
is  an union of log canonical centres of $(X_{\ell+1}, \Gamma_{\ell+1})$ of codimension at most one and
\[
(K_{X_{\ell+1}}+\Gamma_{\ell+1})\vert_{X_\ell} = K_{X_\ell}+\Gamma_\ell.
\] 
In particular, we have 
$(K_X+\Gamma)\vert_{X_\ell} = K_{X_\ell}+\Gamma_\ell$, for any $\ell=1,\dots,m$.  
Then $(X,\Gamma +f^*P_1)$ is log canonical along $f^{-1}(z)$ and $f^{-1}(z)$ is a union of 
log canonical centres of $(X, \Gamma+f^*P_1)$ around $f^{-1}(z)$. 
Let $F$ be the normalisation of $f^{-1}(z)$. Then there exists an induced morphism $j_\ell\colon F\to X_{\ell}$. By adjunction,
it follows that 
$(X_{\ell},\Gamma_\ell+ f^*P_1|_{X_{\ell}})$
 is log canonical for each $\ell=1,\dots,m$ and $j_\ell(F)$ is a log canonical centre of $(X_{\ell},\Gamma_\ell+ f^*P_1|_{X_{\ell}})$.
We may write
\[K_F + \Gamma_F = (K_{X_1}+\Gamma_1+f^*P_1|_{X_1})\vert_F\]
where  $(F, \Gamma_F)$ is log canonical.

We have $f^*P_1\vert_{X_1} = f^*P'_1\vert_{X_1}$, and so it follows that 
\[(K_{X_1}+\Gamma_1+f^*P'_1\vert_{X_1})\vert_F = K_F+\Gamma_F,\]
in particular, adjunction
implies that $(X_1, \Gamma_1+f^*P'_1|_{X_1})$ is log canonical in a neighbourhood of 
$f^{-1}(z)$. By repeatedly applying adjunction again, it follows that $(X_{\ell}, \Gamma_\ell +f^*P'_1|_{X_\ell})$ is log canonical for each $\ell =1,\dots, m$ and, in particular, it follows that $(X, \Gamma+f^*P'_1)$ is log canonical. 
We have that $\Gamma+f^*P'_1=B+H+f^*\overline P$ and so we may replace $P$ by $\overline P$ and the result follows by induction on $k$. 
\end{proof}

\subsection{Maximal moduli}

\begin{definition}\label{d_maximal}
Let $(X/Z,B)$ be a GLC pair induced by a contraction $f\colon X\to Z$. $(X/Z,B)$ is said to have {\bf maximal moduli} if for any GLC pair $(X'/Z,B')$ with induced contraction $f'\colon X'\to Z$ birationally equivalent to $X$ and such that $K_{X'}+B'$ is $f'$-nef  and $(X,B)$ and $(X',B')$ are crepant over the generic point of $Z$, there exists a normal variety $W$ and a $\mathbb R$-divisor $D$ on $W$ such that 
\[
p^*M_X- q^*M_{X'}\sim D\ge 0
\]
where $M_X$ and $M_{X'}$ are the moduli part of $(X/Z,B)$ and $(X'/Z,B')$ respectively and $p\colon W\to X$ and $q\colon W\to Y$ are morphisms which  resolve the indeterminacy of the induced map $X\dashrightarrow X'$. 
\end{definition}

\begin{proposition}\label{p_maximal}
Let $(X/Z,B)$ and $(X'/Z',B')$ be  GLC pairs such that the induced morphisms $f\colon X\to Z$ and $f'\colon X'\to Z'$ are birational equivalent projective contractions. Let $\alpha\colon Z'\dashrightarrow Z$ and $\beta\colon X'\dashrightarrow X$ be the induced birational maps 
and let $p\colon W\to X$ and $q\colon W\to X'$ be a resolution of indeterminacy  of $\beta$. Let $M_X$ and $M_{X'}$  be the moduli part of $(X/Z,B)$ and $(X'/Z',B')$ respectively. 
Assume that 
\begin{enumerate}
\item $(X/Z,B)$ satisfies Property $(*)$ and $B\ge 0$; 
\item $(X'/Z',B')$ is BP stable over $Z'$ and $M_{X'}$ is nef; and 
\item 
$(X,B)$ and $(X',B')$ are crepant over the generic point of $Z$.
\end{enumerate}

Then there exists a $\mathbb R$-divisor $D$ on $W$ such that 
\[p^*M_X- q^*M_{X'}\sim D\ge 0.\]
\end{proposition}

\begin{proof} After possibly replacing $M_X$ and $M_{X'}$ by linearly equivalent divisors, we may assume that $p_*K_W=K_X$ and $q_*K_W=K_{X'}$. 
Write $K_W+B_W=p^*(K_X+B)$ and $K_W+B'_W=q^*(K_{X'}+B')$ and let $B_Z$ and $B'_{Z'}$ be the discriminant of $(X/Z,B)$ and $(X'/Z',B')$ respectively. By Lemma \ref{l_star}, it follows that $B_Z=\Sigma_Z$ is reduced. In particular, $(Z,B_Z)$ is log canonical. 
Let $a\colon Y\to Z$ and $a'\colon Y\to Z'$ be projective birational morphisms from a normal variety $Y$ which resolve the indeterminacy locus of $\alpha\colon Z'\dashrightarrow Z$. As above, we may assume that 
$a_*K_Y=K_Z$ and $a'_*K_Y=K_{Z'}$. 
After possibly replacing $W$ by a higher model, we may assume that the induced map $g\colon W\to Y$ is a morphism. 

Let
\[
A\coloneqq p^*M_X-q^*M_{X'}.
\]
We need to show that $A\ge 0$. By assumption, $A$ is vertical with respect to  $g\colon W\to Y$. 
We write
\[
a^*(K_Z+B_Z)=K_{Y}+B_Y\qquad\text{and}\qquad 
a'^*(K_{Z'}+B'_{Z'}) = K_Y+B'_Y.
\]
Note that, since $(X'/Z',B')$ is BP stable over $Z'$, it follows that $B'_Y$ is the discriminant  of  $(W/Y,B'_W)$. 
Moreover, since $(Z,B_Z)$ is log canonical, so is $(Y,B_Y)$. Note though that $(Y, B_Y)$ is not necessarily the discriminant of $(W/Y, B_W)$.

We have 
\[
\begin{aligned}
A&=(K_W+B_W-g^*(K_Y+B_Y)) - (K_W+B'_W-g^*(K_Y+B'_Y)) \\
&= B_W-g^*B_Y-B'_{W}+g^*B'_Y.
\end{aligned}
\]

Let $E$ be a prime  divisor on $X$ which is vertical with respect to $f$ and let $E'$ be its strict transform on $W$. 
We claim that $m_{E'}A\ge 0$. 
Note that, by the negativity lemma and since $M_{X'}$ is nef, the claim immediately implies the Proposition. 

Let $P\coloneqq f(E)$. We distinguish two cases. If $P$ is contained in $\Sigma_Z$ then, since $(X/Z,B)$ satisfies Property $(*)$, it follows that  $m_EB=1$  and, therefore,
\[
m_{E'}B_W=1.
\]
Since $(X'/Z',B')$ is BP stable over $Z'$, it follows that $(W/Y,B'_W)$ is BP stable over $Y$ and, 
Theorem \ref{t_logstable} implies that it is log stable. 
Thus, since $(Y,B_Y)$ is log canonical and $B'_Y$ is the discriminant  of $(W/Y,B'_W)$, it follows that $(W,B'_W+g^*(B_Y-B'_Y))$ is log canonical and, in particular,
\[
m_{E'}(B'_W+g^*(B_Y-B'_Y))\le 1.
\]
Thus, $m_{E'}A\ge 0$, as claimed. 

Assume now that $P$ is not contained in $\Sigma_Z$.  Lemma \ref{l_star} implies that $f$ is equidimensional outside $\Sigma_Z$ and, therefore, it follows that $P$ is a divisor in $Z$. Let $\Sigma=\Sigma_Z+P$. Then, since $(X/Z,B)$ satisfies Property $(*)$, we have that
$m_{E}(B+f^*(\Sigma-\Sigma_Z))=1$.
Thus, if $P'$ is the strict transform of $P$ in $Y$, we have that
\[m_{E'}(B_W+g^*P')=1.
\]
Since $m_{P'} B'_Y= 1-\gamma_{P'}$ where 
\[
\gamma_{P'}=\sup\{t\in \mathbb R\mid (W,B_W+tg^*P') \text{ is log canonical over the generic point of }P'\},
\]
and since $m_{P'}B_Y=0$, it follows that
\[
\begin{aligned}
m_{E'}A&=m_{E'}(B_W+g^*P'-B'_W-g^*(P'-B'_Y))\\
&= 1-m_{E'}(B'_W+\gamma_{P'}g^*P')\ge 0.
\end{aligned}
\]
Thus, our claim follows again. 
\end{proof}

\begin{remark}\label{r_maximal}
Note that, using the same notation as Proposition \ref{p_maximal}, the same proof as above shows that the same result holds if if we replace (1) by
\begin{enumerate}
\item [(1)']  $(X/Z,B)$ satisfies Property $(*)$ and $f$ is equidimensional.
\end{enumerate}
\end{remark}

\section{Cone theorem for algebraically integrable foliations}

The goal here is to explain a proof of the cone theorem for algebraically integrable foliations. This is the main ingredient in the proof of Theorem \ref{thm_main}. 

Some of these ideas were first considered in \cite{KMM92c} and \cite{Miyaoka93}
in the case of a smooth morphism using the notion of a relative deformation.  
By adopting the foliated viewpoint 
we are able to work on singular varieties as well as working with rational dominant maps.

\subsection{Definitions}
A {\bf foliation of rank $r$} on a normal variety $X$ is a coherent subsheaf $\cal F \subset T_X$ of rank $r$ such that
\begin{enumerate}
\item $T_{\cal F}$ is saturated, i.e. $T_X/T_{\cal F}$ is torsion free; and

\item $T_{\cal F}$ is closed under Lie bracket.
\end{enumerate}
The {\bf co-rank} of $\cal F$  is its co-rank as a subsheaf of $T_X$.  We define the canonical divisor $K_{\cal F}$ to be a divisor on $X$ so that $\cal O_X(K_{\cal F}) \cong (\det T_{\cal F})^*$
and we define the conormal divisor $K_{[X/\cal F]}$ to be a divisor on $X$ so that
$\cal O_X(K_{[X/\cal F]}) \cong (\det (T_X/T_{\cal F}))^*$.
For any positive integer $d$, we denote $\Omega^{[d]}_X\coloneqq (\wedge^d\Omega^1_X)^{**}$.
Then, the foliation $\cal F$ induces a map 
\[
\phi\colon (\Omega^{[r]}_X \otimes \cal O_X(-K_{\cal F}))^{**} \rightarrow \cal O_X.\]
The {\bf singular locus} of $\cal F$, denoted by $\Sing \cal F$, is the cosupport of the image of $\phi$.

\medskip

Let $\sigma\colon Y\dashrightarrow X$ be a dominant map between normal varieties and let $\cal F$ be a foliation of rank $r$ on $X$. We denote by $\sigma^{-1}\cal F$ the {\bf induced foliation} on $Y$ (e.g. see \cite[Section 3.2]{Druel21}).  
If $T_{\cal F}=0$, i.e., $\cal F$ is the foliation by points on $X$, 
then we refer to $\sigma^{-1}\cal F$ as the  {\bf foliation induced by $\sigma$}. 
In this case, the foliation $\sigma^{-1}\cal F$ is called {\bf algebraically integrable}.

Let $f\colon X\to Z$ be a morphism between normal varieties and let $\cal F$ be the induced foliation on $X$. 
If $f$ is equidimensional and $Z$ is smooth, then we define the  {\bf ramification divisor} $R(f)$ of $f$  as 
\[
R(f)=\sum_D (\phi^*D-\phi^{-1}(D))
\]
where the sum runs through all the prime divisors of $Z$. In this case, we have
\[K_{\cal F} \sim K_{X/Z}-R(f)\] 
(e.g. see  \cite[Notation 2.7 and \S 2.9]{Druel15b}).

Let $X$ be a normal variety and let $\cal F$ be a rank $r$ foliation on $X$. 
Let $S\subset X$ be a subvariety. Then  $S$ is said to be  {\bf $\cal F$-invariant}, or {\bf invariant by }$\cal F$, if for any open subset $U\subset X$ and any section $\partial \in H^0(U,\cal F)$, we have that 
\[ \partial (\mathcal I_{S\cap U})\subset \mathcal I_{S\cap U},
\]
where $\mathcal I_{S\cap U}$ denotes the ideal sheaf of $S\cap U$. 
If  $D\subset X$ is a prime divisor then we define $\epsilon(D) = 1$ if $D$ is not $\cal F$-invariant and $\epsilon(D) = 0$ if it is $\cal F$-invariant. Note that if $\cal F$ is the foliation induced by a dominant map $\sigma\colon Y\dashrightarrow Z$, then a divisor $D$ is $\cal F$-invariant if and only if it is vertical with respect to $\sigma$.

\medskip

\subsection{Log canonical foliated pairs}\label{s_lcf}
Let $X$ be a normal variety. 
A {\bf foliated pair} $(\cal F, \Delta)$ on $X$ consists of  a foliation $\cal F$ on $X$ and a $\mathbb R$-divisor 
 divisor $\Delta$
such that $K_{\cal F}+\Delta$ is  $\mathbb R$-Cartier.

Let $X$ be a smooth variety, let $\cal F$ be an algebraically integrable foliation on $X$  and let $\Sigma_X$ be a reduced divisor.
We say that the foliated pair $(\cal F, \Sigma_X)$ is {\bf foliated log smooth} if
\begin{enumerate}
\item $(X, \Sigma_X)$ is a toroidal pair; and

\item there exists a log smooth toroidal pair $(Z, \Sigma_Z)$  and a toroidal contraction $f\colon X\to Z$ which
induces $\cal F$.
\end{enumerate}

Let  $(X,B)$ be a log pair,  $f\colon X\dashrightarrow Y$  a dominant map between normal varieties and let $\cal F$ be the  foliation induced by $f$. Then, by Theorem \ref{t_AK},
there exists a birational morphism  $m\colon X' \rightarrow X$
so that $(m^{-1}\cal F, \Supp m_*^{-1}B+\Exc m)$ is foliated log smooth.
We call such a modification a {\bf foliated log resolution} provided $\Exc m$ is of pure codimension one.

Given a birational morphism $\pi\colon  \widetilde{X} \rightarrow X$ 
and a foliated pair $(\cal F, \Delta)$ on $X$,
let $\widetilde{\cal F}$ be the pulled back foliation on $\tilde{X}$
and $\pi_*^{-1}\Delta$ be the strict transform of $\Delta$ in $\widetilde X$. 
We may write
\[
K_{\widetilde{\cal F}}+\pi_*^{-1}\Delta=
\pi^*(K_{\cal F}+\Delta)+ \sum a(E, \cal F, \Delta)E.\]
where $\pi_*K_{\widetilde{\cal F}}=K_{\cal F}$,  the sum runs over all the prime $\pi$-exceptional divisors on $\widetilde X$ and  the rational number $a(E,\cal F,\Delta)$ is called the {\bf discrepancy} of $(\cal F,\Delta)$ with respect to $E$. 
We say that  the foliated pair $(\cal F, \Delta)$ is {\bf log canonical}
 if 
$a(E, \cal F, \Delta) \geq -\epsilon(E)$, 
for any birational morphism  $\pi\colon \widetilde X\to X$ and for any prime $\pi$-exceptional divisor $E$ on  $\widetilde X$.
We say that $(\cal F, \Delta)$ is {\bf F-dlt} if there exists a 
foliated log resolution which only extracts divisors $E$ 
of discrepancy $a(E,\cal F,\Delta)>-\epsilon(E)$.

Let $W\subset X$ be  a subvariety. We say that $W$ is a {\bf log canonical centre} of $(\cal F,\Delta)$ if $(\cal F,\Delta)$ is log canonical at the generic point of $W$ and there exists a birational morphism $\pi\colon Y\to X$ and a prime divisor $E$ on $Y$ of discrepancy $a(E,\cal F,\Delta)=-\epsilon(E)$ and whose centre in $X$ is $W$. The divisor $E$ is called a {\bf log canonical place} of $(\cal F,\Delta)$. 
We say that $W$ is a {\bf non log canonical centre} of $(\cal F,\Delta)$ if there exists a birational morphism $\pi\colon \tilde X\to X$ and a prime divisor $E$ on $\tilde X$ of discrepancy $a(E,\cal F,\Delta)<-\epsilon(E)$ and whose centre in $X$ is $W$. 
We denote by $\nlc(\cal F,\Delta)$  the set of all the non log canonical centres of $(\cal F,\Delta)$.

\begin{lemma}
\label{lem_toroidal_lc}
Let $(X, \Sigma_X)$ and $(Z, \Sigma_Z)$ be toroidal pairs, let $f\colon X \rightarrow Z$ be a toroidal contraction 
and let $\cal F$ be the induced foliation.
Let $\Delta$ be the horizontal part of $\Sigma_X$.

Then $(\cal F, \Delta)$ is log canonical.
\end{lemma}
\begin{proof}
The statement is local about a point $x \in X$. Thus, we may freely assume that $(X, \Sigma_X)$ is a toric pair, $x$ is the unique closed orbit 
and $\cal F$ is a torus invariant foliation. 

If $\cal F$ is of co-rank $q$ then  $\cal F$
can be defined by a logarithmic torus invariant $q$-form $\omega = \sum_{|I| = q} \alpha_I \frac{dx_I}{x_I}$, for some $\alpha_I \in \mathbb C$.
Let $D$ denote the polar locus of $\omega$, and observe that $\cal O_X(K_{[X/\cal F]} )= \cal O_X(-D)$. Note that $D$ is
$\cal F$-invariant and torus invariant.  In particular, since no component of $\Delta$ is invariant,
it follows that $D$ and $\Delta$ have no components in common.

Let $b\colon X' \rightarrow X$ be a birational morphism, set $\cal F' = b^{-1}\cal F$ and let $\omega' = b^*\omega$.  
Note that 
$\omega' = b^*\omega$ is a logarithmic $q$-form with polar locus equal to
$b_*^{-1}D+\sum \delta_i E_i$ where $E_i$ are the $b$-exceptional divisors
and $\delta_i \in \{0, 1\}$.  In fact, observe that $\delta_i = 1-\epsilon(E_i)$.
So we have $K_{[X'/\cal F']} = b_*^{-1}K_{[X/\cal F]}+\sum (\epsilon(E_i)-1)E_i$.

Thus, if we write \[K_{X'}+b_*^{-1}(D+\Delta) = b^*(K_X+D+\Delta)+F\] and \[K_{\cal F'}+b_*^{-1}\Delta = b^*(K_{\cal F}+\Delta)+G\]
then we see that $G = F+\sum (1-\epsilon(E_i)) E_i$.
Our claim then follows by observing that $(X, D+\Delta)$ is log canonical.
\end{proof}

\subsection{Adjunction results}
\label{s_adjunction}

\begin{proposition}
\label{prop_inv_adj}
Let $X$ be a normal $\mathbb Q$-factorial variety and let $\cal F$ be a foliation induced
by a  contraction $f\colon X \rightarrow Z$.  Let $T$ be a prime divisor on $X$
with normalisation $\nu\colon S \rightarrow T$.
Let $(\cal F,\Delta)$ be a foliated  pair on $X$ where $\Delta \geq  0$ is such that  $m_T\Delta = \epsilon(T)$.

Then 

\begin{enumerate}

\item there exists a foliated  pair $(\cal F_S,\Delta_S)$ on $S$ 
such that $\Delta_S\ge 0$ and 
\[
\nu^*(K_{\cal F}+\Delta) = K_{\cal F_S}+\Delta_S; 
\]

\item if $(\cal F,\Delta)$ is log canonical, then $(\cal F_S,\Delta_S)$ is also log canonical; and 


\item if $\epsilon(T) = 1$ and $(\cal F, \Delta)$ is F-dlt then $(\cal F_S, \Delta_S)$ is F-dlt.
\end{enumerate}
\end{proposition}

The induced foliation $\cal F_S$ on $S$ is called the {\bf restricted foliation}. 
Observe that if $S$ is $\cal F$-invariant then $\cal F_S$ is not necessarily the foliation induced by 
$S \rightarrow Y$ where $Y$ is the normalisation of  $f(T)$. For example, consider the blow up at a point of a smooth foliation induced by a smooth morphism from  a threefold to a surface and let $S$ be the exceptional divisor.  However, it is still the case that $\cal F_S$ is induced by a rational dominant map.

\begin{proof}
We argue in cases based on $\epsilon(T)$.

\medskip

{\bf Case I: $\epsilon(T) = 0$.}
Let $U\subset X$ be an open subset containing the general point of $T$ and not intersecting
$\Sing \cal F\cup \Sing X\cup \Sing T$. Let $T'=T\cap U$. Then, since $T$ is invariant, we have that, over $U$, the inclusion of sheaves 
$T_{\cal F}|_{T'}\to {T_X}|_{T'}$ factors through $T_{T'}$ and, in particular, it defines a foliation on $T'$ which yields a foliation $\cal F_S$ on $S$. The existence of $\Delta_S$ follows easily from \cite[Lemma 3.7]{AD14}. Thus, (1) follows. 

We now prove (2). 
Assume that $(\cal F,\Delta)$ is log canonical. 
Let $m\colon X' \rightarrow X$ be a log resolution of $(\cal F, \Delta)$ and let $S'$ be the strict transform of $T$. We may write
\[
K_{\cal F'}+\Delta'+F_0+F_1 = m^*(K_{\cal F}+\Delta)
\]
where $\Delta' = m_*^{-1}\Delta$ and $F_0,F_1$ are $\mathbb R$-divisors on $X'$ such that  $\Supp F_0$ is $\cal F'$-invariant and $\Supp F_1$ is not $\cal F'$-invariant. Since $(\cal F,\Delta)$ is log canonical, we have that $F_0\le 0$ and the coefficients of $F_1$ are at most one. 
We define
\[
K_{\cal F'}\vert_{S'} = K_{\cal F_{S'}}+\Theta_{S'}\qquad\text{and}\qquad 
\Delta_{S'} \coloneqq \Theta_{S'}+\Delta'\vert_{S'}.
\]  
We may replace $m$ by a higher model so that  $(\cal F_{S'}, \Delta_{S'})$ is also log smooth. Let $n\colon S'\to S$ be the induced morphism. We have
\[
K_{\cal F_{S'}}+\Theta_{S'}+\Delta'|_{S'} + F_0|_{S'}+F_1|_{S'} = (K_{\cal F'}+\Delta'+F_0+F_1)|_{S'}=n^*(K_{\cal F_{S}}+\Delta_S).
\]
We now make note of some easy consequences of log smoothness of $(\mathcal F', \Delta'+\Exc m)$.  Indeed, each of these claims
is local, and so we may assume in fact that $(\mathcal F', \Delta'+\Exc m)$ is induced by a toric morphism in which case these claims are easy to verify.
First, we note that $\Theta_{S'}$ is a reduced divisor supported
exactly on the codimension one components of $\Sing \cal F'\cap S'$. 
Next, $\Delta'+F_1$ contains no codimension $2$ components of $\Sing \cal F'$ and so
$\Theta_{S'}$ and $(\Delta'+F_1)\vert_{S'}$ have
no components in common.
Finally, no component of
$F_1\vert_{S'}$ is $\cal F_{S'}$-invariant
and no component of $ \Sing \cal F'\cap S'$ is $\cal F_{S'}$ invariant.
From these three claims we see that $(\cal F_{S'}, \Theta_{S'}+\Delta'|_{S'} + F_0|_{S'}+F_1|_{S'})$ is log canonical
and so $(\cal F_{S}, \Delta_S)$ is log canonical from which (2) follows.



\medskip

{\bf Case II: $\epsilon(T) = 1$.}
Note that  in this case  $T$ dominates $Z$, let $g\colon S \to Z$ be the induced morphism and let $\cal F_S$ be the induced foliation.
Let $H$ be a general ample divisor on $X$, let $h\colon H \rightarrow Z$ be the induced morphism and let $\cal F_H$
be the induced foliation.

We now show that $\Delta_S \geq 0$ by induction on $\dim (X/Z)$.
If $\dim (X/Z) = 1$ then
 we may apply \cite[Proposition 2.19]{CS20} to see that $\nu^*(K_{\cal F}+T) = D \geq 0$
(note that the threefold hypothesis there is not necessary) from which we may conclude.
Observe that if $H$ is sufficiently general and if we write $(K_{\cal F}+H))\vert_H = D_H \geq 0$
then $D_H$ is supported on the locus where $H$ is tangent to $\cal F$.  In particular, if $H$ is sufficiently general
and $E \subset H$ is a divisor then $m_ED_H = 0$ if $h_*E = 0$ and $m_ED_H = r_E-1$ if $h_*E \neq 0$
where $r_E$ is the ramification index of $h$ along $E$.

Suppose now that $\dim (X/Z) >1$.
By \cite[Proposition 3.6]{Druel21} we see that $(K_{\cal F}+H)\vert_H = K_{\cal F_H}$.
Let $S_H \rightarrow T\cap H$ be the normalisation morphism.  
Write $(K_{\cal F}+\Delta+H)\vert_H = K_{\cal F_H}+\Delta\vert_H$ and observe that
$\epsilon(T\cap H) = 1$.
By induction if we write $(K_{\cal F_H}+\Delta\vert_H)\vert_{S_H} = K_{\cal F_{S_H}}+\Delta_{S_H}$
then $\Delta_{S_H} \geq 0$.  On the other hand, we have $(K_{\cal F}+\Delta+H)\vert_S = K_{\cal F_S}+\Delta_S+H_S$
where $H_S\cong S_H$.  If $\dim (S/Z) >1$ then again by \cite[Proposition 3.6]{Druel21} we have $\Delta_{S_H} = \Delta_{S}\vert_{H_S}$
and so $\Delta_S \geq 0$.  If $\dim (S/Z) = 1$ then we have 
$(K_{\cal F_S}+\Delta_S+H_S)\vert_{H_S} = K_{\cal F_{S_H}}+\Delta_S\vert_{H_S}+D$
where $D \geq 0$ is supported on the ramification divisor of $H_S \rightarrow Z$.  In this case 
$0 \leq \Delta_{S_H} = \Delta_S\vert_{H_S}+D$.
By choosing $H$ to be sufficiently general we may arrange it so that $D$ and $\Delta_S\vert_{H_S}$ have no components in common
from which we may conclude that $\Delta_S \geq 0$.

(2) and (3) follow by similar arguments as in the previous case.
\end{proof}

\subsection{Tangent subvarieties}

Let $X$ be a normal variety and let $\cal F$ be a foliation induced by a rational map $X \dashrightarrow Z$.
We say a subvariety $V \subset X$ is {\bf tangent} to $\cal F$ if there exist
\begin{enumerate}
\item a birational modification $\mu\colon X' \rightarrow X$ and an equidimensional contraction
$X' \rightarrow Z'$ which induces $\mu^{-1}\cal F$; and

\item a subvariety $V' \subset X'$ contained in a fibre of $X' \rightarrow Z'$ such that $V = \mu(V')$.
\end{enumerate}

We remark that when $X$ and $\cal F$ are smooth this agrees with the usual definition of tangency.
We collect some basic observations about tangent subvarieties below.

\begin{lemma}
\label{lem_tangency}
Let $X$ be a normal variety and let $\cal F$ be a foliation induced by a rational map $X \dashrightarrow Z$.
\begin{enumerate}
\item Let $\pi\colon Y \rightarrow X$ be a birational morphism and let $W \subset Y$ be a subvariety
tangent to $\pi^{-1}\cal F$.  Then $\pi(W)$ is tangent to $\cal F$.

\item Let $\pi\colon Y \rightarrow X$ be a birational morphism and let $V \subset X$ be tangent to $\cal F$.
Then there exists a subvariety $W \subset Y$ tangent to $\pi^{-1}\cal F$ and such that $\pi(W) = V$.

\item Suppose that $\cal F$ is induced by a morphism $f\colon X \rightarrow Z$ and suppose that $V$ is tangent to $\cal F$.
Then $V$ is contained in a fibre of $X \rightarrow Z$.

\item Suppose that $\cal F$ is induced by a morphism $f\colon X \rightarrow Z$ and let $\nu\colon S \rightarrow T \subset X$
be the normalisation of a divisor with restricted foliation $\cal F_S$. If $V\subset S$ is tangent to $\cal F_S$ then $\nu(V)$
is tangent to $\cal F$.
\end{enumerate}

\end{lemma}

\begin{remark} As above, 
consider the blow up at a point of  a smooth threefold which admits a smooth morphism to a surface, let $\cal F$ be the induced foliation  and let $S$ be the exceptional divisor. 
Then the restricted foliation $\cal F_S$ is the radial foliation on $\mathbb P^2$. A curve in $S$  which is not $\cal F_S$-invariant is also not tangent to $\cal F$, even though it is contained in a fibre of the morphism inducing $\cal F$. Indeed, the induced morphism is not equidimensional. \end{remark}

\begin{proof}
(1) is a straightforward consequence of the definition of tangency.

Now, let $\mu\colon X' \rightarrow X$ be a modification
so that $\mu^{-1}\cal F$ is induced by an equidimensional contraction $f'\colon X' \rightarrow Z'$ and there exists a subvariety $V' \subset X'$
contained in a fibre of $f'$ so that $\mu(V') = V$. 

We now prove (2). Let $\nu\colon X'' \rightarrow X'$ be a birational morphism so that $\nu^{-1}\mu^{-1}\cal F$ is induced by an equidimensional
contraction $f''\colon X'' \rightarrow Z''$ and so that we have a morphism $p\colon X'' \rightarrow Y$.
Observe that we may find a subvariety $V'' \subset X''$ contained in a fibre of $f''$ so that $\nu(V'') = V'$ and so we may take
$W = p(V'')$, as required.

We now prove  (3). 
We may assume that there exists a birational morphism $\sigma\colon Z' \rightarrow Z$ such that $\sigma\circ f'=f\circ \mu$. 
Since $V'$ is contained in a fibre of $f'$ and $\sigma(f'(V')) = f(V)$ it follows immediately that $V$ is contained in a fibre of $f$. 

We finally prove (4). 
Let $\mu\colon X' \rightarrow X$ be a modification
so that $\mu^{-1}\cal F$ is induced by an equidimensional contraction $f'\colon X' \rightarrow Z'$. 
By replacing $X'$ with a higher model, we may freely assume 
that we have a morphism $\tau\colon  T'\coloneqq \mu_*^{-1}T \rightarrow S$. 
By (2) there exists $W \subset T'$ which is tangent to $\tau^{-1}\cal F_{S}$ and such that $\tau(W) = V$.
By (3) we know that $W$ is contained in a fibre of $T' \rightarrow f'(T')$, and so $W$ is contained in a fibre
of $X' \rightarrow Z'$, which implies that $\mu(W) = \nu(V)$ is tangent to $\cal F$.
\end{proof}

\subsection{Property $(*)$ foliations and existence of Property $(*)$ modifications}

\begin{definition}\label{d_fstar}
Let $X$ be a normal variety and let $(\cal F,\Delta)$ be a foliated  pair on $X$. We say that $(\cal F, \Delta)$ {
\bf satisfies Property $(*)$} if there exists a projective contraction $f\colon X \rightarrow Z$ such that $\cal F$ is induced by $f$
and an $\cal F$-invariant divisor $G\ge 0$ on $X$  
such that $(X/Z, \Delta+G)$ satisfies Property $(*)$ (see Definition \ref{d_property*}). The divisor $G$ is said to be {\bf associated to} $(\cal F,\Delta)$. 
\end{definition}

The importance of Property $(*)$ is explained by the following result.

\begin{proposition}
\label{prop_*comparison}
Let $(\cal F, \Delta)$ be a foliated pair on a normal variety $X$ which satisfies Property $(*)$ and 
let $f\colon X \rightarrow Z$ be the induced contraction. Let $G \geq 0$ be the $\cal F$-invariant  divisor associated to $(\cal F,\Delta)$ and let $M_X$ be the moduli part of $(X/Z,\Delta+G)$. 
Suppose moreover that $f\colon X \rightarrow Z$ is equidimensional.

Then 
\begin{enumerate}
\item $K_{\cal F}+\Delta \sim M_X$; and

\item $K_{\cal F}+\Delta \sim_{f, \mathbb R} K_X+\Delta+G$.
\end{enumerate}
\end{proposition}
\begin{proof}
Let $S \coloneqq (f_*R(f))_{red}$ and observe that we have $f^*S = (f^*S)_{red}+R(f)$
and so $K_{\cal F}+f^*S = K_{X/Z}+(f^*S)_{red}$.
Set $B\coloneqq \Delta+G$. By assumption, $(X/Z, B)$ satisfies Property $(*)$. Let $B_Z$ be its discriminant. By Lemma \ref{l_star}, it follows that $B_Z$ is reduced. 

To prove (1) its suffices to show that $D-f^*D_Z = \Delta-R(f)$.
Pick a divisor $E$ on $Z$ and let $E'$ be a component of $f^*E$.
By Property $(*)$, if $E$ is a component of $S$ then $E$ is a component
of $B_Z$ and we have  $m_E B_Z=m_{E'}B =1$ and so $m_{E'}(B-f^*B_Z) = m_{E'}R(f)$.
If $E$ is not a component of $S$ then $m_{E'}(B-f^*B_Z) = 0$.  We may conclude by observing that
$B-\Delta$ is supported on vertical divisors. Thus, (1) follows. 
 
Finally, (2) is a direct consequence of (1).
\end{proof}

\begin{proposition}\label{p_canonical}
Let $(\cal F, \Delta)$ be a foliated pair on a normal variety $X$ which satisfies Property $(*)$ and which is induced by an equidimensional projective contraction 
 $f\colon X \rightarrow Z$. Let $G \geq 0$ be the $\cal F$-invariant  divisor associated to $(\cal F,\Delta)$ and let $B=\Delta+G$.

Then  $(X/Z,B)$ is BP semi-stable over $Z$ (cf. Definition \ref{d_bpstable}) if and only if $(\cal F,\Delta)$ admits log canonical singularities. 
\end{proposition}

\begin{proof}
We fix a birational morphism  $X'\to X$ (resp. $Z'\to Z$). 
By Proposition \ref{prop_*resolution}, 
after possibly replacing $X'$ (resp. $Z'$) by a higher model, we may assume that 
there exist an equidimensional  
contraction $f'\colon X'\to Z'$ which is birationally equivalent to $f$  
 and a $\mathbb R$-divisor $B'$ on $X'$ such that 
the induced maps $\alpha\colon Z'\to Z$ and $\beta\colon X'\to X$ are birational morphisms,
 $(X'/Z',\overline B)$ satisfies Property $(*)$ and  $(X,B)$ and $(X',\overline B)$ are crepant over the generic point of $Z$. 
In particular, if $\Delta'$ is the horizontal part of $\overline B$ and $G'$ is its vertical part, then $(\cal F',\Delta')$ is a foliated pair satisfying Property $(*)$ and such that $G'$ is its associated divisor.  
Moreover, if we write 
\[
K_{X'}+B'=\beta^*(K_X+B)
\]
then, it follows that $\overline B-B'$ is an effective $\beta$-exceptional $\mathbb R$-divisor which is $\cal F'$-invariant. By Proposition \ref{prop_*comparison}, we may assume that
$K_{\cal F}+\Delta$ (resp.   $K_{\cal F'}+\Delta'$) coincides with  the moduli part $M_X$ (resp.  $M_{X'}$) of $(X/Z,B)$ (resp. $(X'/Z',\overline B)$).
Let $B_Z$ and $\overline B_{Z'}$ be the discriminant of $(X/Z,B)$ and $(X'/Z',\overline B)$ respectively and let $B'_Z$ be defined by 
\[
K_{Z'}+B'_Z=\alpha^*(K_Z+B_Z).
\]
Since  $(X''/Z'',\overline B)$ satisfies Property $(*)$, by Lemma \ref{l_star}, it follows that $\overline B_Z$ coincides with the induced divisor $\Sigma_{Z'}$ 
 on $Z'$. 
 Let $A\coloneqq M_{X'}- \beta^*M_X$. Since 
\[A=\overline B-B' -f'^*(\overline B_{Z'}-B'_Z),\]
it follows that the support of $A$ is $\cal F'$-invariant. We may assume that the image of the support of $A$ in $Z'$ is contained in $\Sigma_{Z''}$. 

\medskip 

We first prove that $A\geq 0$ if and only if $a(E,X,\mathcal F) \ge -\epsilon(E)$ for every $\beta$-exceptional prime divisor $E$.
Indeed, 
let 
\[
A'=K_{\cal F'}+\beta_*^{-1}\Delta+\sum E- \beta^*(K_{\cal F}+\Delta)
\]
where the sum runs over all the $\beta$-exceptional prime divisors which are not $\cal F'$-invariant.
 Then, since $(X,B)$ is log canonical and since $A=K_{\cal F'}+\Delta'-\beta^*(K_{\cal F'}+\Delta)$, it follows that $A'\ge A$. Moreover $A'\ge 0$ if and only if $a(E,X,\mathcal F) \ge -\epsilon(E)$ for every $\beta$-exceptional prime divisor $E$, and, by construction, we have that  $m_EA=m_EA'$ for any $\cal F'$-invariant prime divisor $E$ on $X'$. Since the support of $A$ is $\cal F'$-invariant, our claim follows. 

\medskip

Let $B'_{Z'}$ be the discriminant of $(X'/Z',B')$. We now prove that $A\ge 0$ if and only if $B'_{Z'}\le  B'_{Z}$. Note that the claim immediately implies our result. 
Let $P$ be a component of $\Sigma_{Z'}$. Then, for any prime divisor  $E$ contained in $f^{-1}(P)$, we have that $m_E \overline B=1$. Thus, $m_E(A)\ge 0$ if and only if 
\[
m_E(B'+f'^*(\overline B_{Z'}-B'_Z))\le 1.
\]
It follows that $m_E(A)\ge 0$ for any prime divisor $E$ contained in $f^{-1}(P)$ if and only if $(X',B'+f'^*(\overline B_{Z'}-B'_Z))$ is log canonical over the generic point of $P$. 
Since $m_P(\overline B_{Z'}-B'_Z))=1-m_PB'_Z$, this is in turn equivalent to the inequality 
\[
\gamma_P\ge 1-m_PB'_Z
\]
where 
\[
\gamma_P=\sup\{t\in \mathbb R\mid (X',B'+tf'^*P) \text{ is log canonical over the generic point of }P\}.
\]
Since $m_PB'_{Z'}=1-\gamma_P$, our claim follows. 
\end{proof}

\begin{definition}\label{d_*mod}
Let $X$ be a normal variety and let $(\cal F,\Delta)$ be a foliated pair where $\cal F$ is the  foliation induced by a rational
map $f\colon X \dashrightarrow Y$ and such that none of the components of $\Delta$  is $\cal F$-invariant. 
We define a {\bf Property $(*)$ modification} for $(\cal F,\Delta)$ to be a birational morphism
$\pi\colon X' \rightarrow X$ such that 
if $\cal F'\coloneqq \pi^{-1}\cal F$ and $\Delta'\coloneqq \pi_*^{-1}\Delta+\sum \epsilon(E)E$, 
where the sum is taken over all the $\pi$-exceptional divisors, then  
the following hold
\begin{enumerate}
\item  $X'$ is klt;
\item $\cal F'$ is induced by an equidimensional morphism $X'\to Z'$; 
\item $(\cal F' , \Delta')$ is log canonical and it satisfies Property $(*)$; and
\item $(K_{\cal F'}+\Delta') +F = \pi^*(K_{\cal F}+\Delta)$ where $F \geq 0$ is a $\mathbb R$-divisor on $X'$ whose image in $X$ is contained in the union of non log canonical centres of $(\cal F,\Delta)$. 
\end{enumerate}
\end{definition}

\subsection{Cone theorem}

The goal of this section is to prove the cone theorem for algebraically integrable foliations. To this end, we proceed by induction on the dimension of the ambient variety, by showing, at the same time,  the existence of a Property $(*)$ modification.

\begin{theorem}[Cone Theorem]\label{thm_cone} 
Let $(\cal F,\Delta)$ be a foliated pair on  a normal projective variety 
$X$ of dimension $n$, where $\cal F$ is the foliation induced
by a dominant map $f\colon X \dashrightarrow Y$ between normal varieties and $\Delta\ge 0$.

Then we may write
\[\overline{NE}(X) = \overline{NE}(X)_{K_{\cal F}+\Delta \geq 0}+Z_{-\infty}+\sum \mathbb R_+[\xi_i]\]
where $\xi_i$ is a rational curve tangent $\cal F$
with $0 \leq -(K_{\cal F}+\Delta)\cdot \xi_i \leq 2n$
and $Z_{-\infty}$ is the contribution from $\nlc {(\cal F, \Delta)}$.
\end{theorem}

\begin{theorem}[Existence of Property $(*)$ modifications]
\label{thm_prop*}
Let $(\cal F,\Delta)$ be a foliated pair on  a normal projective variety 
$X$ of dimension $n$, where $\cal F$ is the foliation induced
by a dominant map $f\colon X \dashrightarrow Y$ between normal varieties and $\Delta\ge 0$.

Then  $(\cal F, \Delta)$ admits a Property $(*)$ modification $\pi\colon X'\to X$. 

Moreover, for any log canonical centre $W$  of $(\cal F,\Delta)$,  we can take $\pi$ so that 
if $\cal F'=\pi^{-1}\cal F$ and $\Delta'\coloneqq \pi_*^{-1}\Delta+\sum \epsilon(E)E$, 
where the sum is taken over all the $\pi$-exceptional divisors, then 
 there exists a codimension one log canonical centre $S$ of $(\cal F',\Delta')$ such that $W= \pi(S)$.
\end{theorem}

In order to prove Theorem \ref{thm_cone} and Theorem \ref{thm_prop*}, we will prove the following

\begin{lemma}\label{l_prop*}
Assume Theorem \ref{thm_cone} in dimension $\leq n-1$.

Then Theorem \ref{thm_prop*} holds in dimension $n$. 
\end{lemma}

\begin{lemma}
\label{l_cone} 
Assume Theorem \ref{thm_prop*} in dimension $\leq n$. 

Then Theorem \ref{thm_cone} holds in dimension $n$. 
\end{lemma}

We first prove the following

\begin{lemma}
\label{lem_special_cone}
Assume Theorem \ref{thm_cone} in dimension $\leq n-1$.

Let $(\cal F,\Delta)$ be a log canonical foliated pair on  a normal projective $\mathbb Q$-factorial klt variety 
$X$ of dimension $n$, where $\cal F$ is the foliation induced
by a projective contraction $f\colon X \to Y$ between normal varieties and $\Delta\ge 0$. Assume that  $(\cal F, \Delta)$ satisfies Property $(*)$.
Suppose moreover that for any $(K_{\cal F}+\Delta)$-negative extremal ray, there exists
a divisor $E$ such that $R$ is contained in the image of $\overline{NE}(E) \rightarrow \overline{NE}(X)$
and such that $m_E\Delta = \epsilon(E)$.

Then we may write
\[\overline{NE}(X) = \overline{NE}(X)_{K_{\cal F}+\Delta \geq 0}+\sum_{i\ge 1} \mathbb R_+[\xi_i]\]
where $\xi_1,\xi_2,\dots$ are rational curves tangent to $\cal F$
with $0 \leq -(K_{\cal F}+\Delta)\cdot \xi_i \leq 2\dim X$. 

Moreover, if $C$ is a curve such that $[C] \in \mathbb R_+[\xi_i]$ then $C$ is contracted by $f$.
\end{lemma}
\begin{proof}
The existence of the rational curves $\xi_1,\xi_2,\dots$ and the claim on the structure of the cone 
is a direct consequence of adjunction (cf. Proposition \ref{prop_inv_adj}) and 
Theorem \ref{thm_cone} applied to the restricted foliation on $E$.
Fix $i\ge 1$. The tangency of  $\xi_i$ follows from (4) of Lemma \ref{lem_tangency}  and (3) of Lemma \ref{lem_tangency} 
implies that $\xi_i$ is contracted by $f$. 
Our final claim follows by observing that if  $[C] \in \mathbb R_+[\xi_i]$ then for any ample divisor $A$ on $Y$ we have
$C\cdot f^*A = \xi_i\cdot f^*A = 0$ and so $C$ is contracted by $f$.
\end{proof}

\begin{lemma}
\label{lem_special_contraction}
Assume Theorem \ref{thm_cone} in dimension $\leq n-1$.

Let $(\cal F,\Delta)$ be a log canonical foliated pair on  a normal projective $\mathbb Q$-factorial klt variety 
$X$ of dimension $n$, where $\cal F$ is the foliation induced
by an equidimensional projective contraction $f\colon X \to Y$ between normal varieties and $\Delta\ge 0$. Assume that  $(\cal F, \Delta)$ satisfies Property $(*)$.
Suppose moreover that for any $(K_{\cal F}+\Delta)$-negative extremal ray, there exists
a divisor $E$ such that $R$ is contained in the image of $\overline{NE}(E) \rightarrow \overline{NE}(X)$
and such that $m_E\Delta = \epsilon(E)$. Let $R$ be a $(K_{\cal F}+\Delta)$-negative extremal ray.


Then the contraction and, if the contraction is small, the flip associated to $R$ exist.
\end{lemma}
\begin{proof}
By Lemma \ref{lem_special_cone}, $R$ is spanned by a rational curve $\xi$ which is contracted by $f$.
Let  $G \geq 0$ be the  divisor associated to $(\cal F,\Delta)$ (cf. Definition \ref{d_fstar}). 
By Proposition \ref{prop_*comparison} we see that $K_{\cal F}+\Delta \sim_{f, \mathbb R} K_X+\Delta+G$.
In particular, $(K_{\cal F}+\Delta)\cdot \xi = (K_X+\Delta+G)\cdot \xi$.
Lemma \ref{l_star} implies that  $(X, \Delta+G)$ is log canonical.
Thus, by the base point free Theorem and  \cite{BCHM06} the contraction and flip associated to $R$ exists.
\end{proof}

%

\begin{proof}[Proof of Lemma \ref{l_prop*}]
By Theorem \ref{t_AK} and Proposition \ref{p_toroidal}
there exists a birational morphism  $m\colon W \rightarrow X$ 
so that $m^{-1}\cal F$ is induced 
by an equidimensional contraction $g\colon W \rightarrow Z_W$ and such that  $(W/Z_W, \Delta_W+G)$
satisfies Property $(*)$ where  $G \geq 0$ is supported on $\cal F$-invariant divisors and $\Delta_W\coloneqq m_*^{-1}\Delta+\sum \epsilon(E)E$, where the sum runs over
all the $m$-exceptional divisors.  
Note that $(m^{-1}\cal F, \Delta_W)$ is not necessarily foliated log smooth, but by Lemma \ref{lem_toroidal_lc},
it is log canonical.
 
Let $\cal F_W \coloneqq m^{-1}\cal F$. We claim we may run a $(K_{\cal F_W}+\Delta_W)$-MMP over $X$.
Indeed, Lemma \ref{lem_special_cone} and Lemma \ref{lem_special_contraction} imply that 
all the required contractions and flips exist.  By Lemma \ref{lem_special_cone}, it follows that every curve contracted by this MMP is also mapped to a
point in $Z_W$ and, in particular, this MMP is also an MMP over $Z_W$.
By Proposition \ref{prop_MMP_preserves_*}, this MMP preserves Property $(*)$ and the fact that the induced contraction is equidimensional.
Since each step of this MMP is also a step in a relative $(K_W+\Delta_W+G)$-MMP over $X$, it follows that
 this MMP terminates. It is easy to check that the output of this MMP is a Property  $(*)$ modification for $(\cal F,\Delta)$.

Let $W$ be a log canonical centre of $(\cal F,\Delta)$. By passing to a higher model which extracts a log canonical place whose centre is $W$
and running an MMP as above,  we can guarantee that  there exists a component $S$ of $\rfdown \Delta'.$ such that $W= \pi(S)$.  
\end{proof}

\begin{proof}[Proof of Lemma \ref{l_cone}]
We may freely assume that Theorem \ref{thm_cone} holds in dimension $\leq n-1$.

Let $R \subset \overline{NE}(X)$ be a $(K_{\cal F}+\Delta)$-negative exposed extremal ray
and let $A$ be an appropriate choice of an ample divisor
so that $H_R\coloneqq K_{\cal F}+\Delta+A$ is a supporting hyperplane to $R$
and such that $H_R$ is positive on $\overline{NE}(X)\setminus R$.

Let $\nu$ be the numerical dimension of $H_R$. We assume first that $\nu<n$, i.e. $H_R$ is not big. We define
$D_i=H_R$ for $1\le i\le \nu+1$ and $D_i=A$ for $\nu+1<i \le n$.  Then 
\[
D_1\cdot \ldots\cdot D_n=H_R^{\nu+1}\cdot A^{n-\nu-1}=0
\]
and 
\[
-(K_{\cal F}+\Delta)\cdot D_2\cdot\ldots\cdot D_n= (A-H_R)\cdot H_R^{\nu}\cdot A^{n-\nu-1}>0.
\]
Fix a sufficiently large number $k\ge 0$ such that $M\coloneqq kH_R-(K_{\cal F}+\Delta)$ is ample. 
Thus, we may apply \cite[Corollary 2.28]{Spicer20} and
 through a general point of $X$ there is a rational curve $\xi$ which is tangent to $\cal F$, such that $H_R\cdot \xi=0$ and 
 \[
 -(K_{\cal F}+\Delta)\cdot \xi=M\cdot \xi \le 2n \frac {M\cdot H_R^\nu\cdot A^{n-\nu-1}}{-K_{\cal F}\cdot H_R^\nu\cdot A^{n-\nu-1}}=2n\frac
  {-(K_{\cal F}+\Delta)\cdot H_R^\nu\cdot A^{n-\nu-1}}{-K_{\cal F}\cdot H_R^\nu\cdot A^{n-\nu-1}}\le 2n.
 \] 

\medskip 

Suppose now that $H_R$ is big. Then there exists a prime divisor $S\subset X$ such that $R$ is $S$-negative. 
Consider the set $\Lambda$ of tuples $(W, \lambda)$ where
\begin{enumerate}
\item  $\lambda \geq 0$; 
\item $W$ is a log canonical centre of $(\cal F, \Delta+\lambda S)$ and 
\item $R \subset \text{im}(\overline{NE}(W^\nu) \rightarrow \overline{NE}(X))$
where $W^\nu$ is the normalisation of $W$ and $W^\nu \rightarrow X$ is the induced morphism.
\end{enumerate}
This set is non-empty since $S$ is $R$-negative. 
Let $(W_0, \lambda_0)$ be a minimal element in this set with respect
to inclusion on the first entry.

Let $\mu\colon \overline{X} \rightarrow X$
be a Property $(*)$ modification of $(\cal F, \Theta \coloneqq \Delta+\lambda \Delta)$
such that $\mu^{-1}(W_0)$ is a divisor, whose existence is guaranteed by Theorem \ref{thm_prop*} in dimension $n$.
Let  $\overline{\Theta} = \mu_*^{-1}\Theta+\sum \epsilon(E)E$, where the sum runs over the  $\mu$-exceptional divisors. Then 
we may write 
\[
K_{\overline{\cal F}}+\overline{\Theta}+F = \mu^*(K_{\cal F}+\Theta)
\]
for some $F \geq 0$.

Note that, for every component $F_0$ of $F$, we have $\mu(F_0) \subset \nlc {(\cal F, \Theta)}$. There exists a $(K_{\overline{\cal F}}+\overline{\Theta})$-negative
extremal ray $R'$ such that $\mu_*R' = R$.  By minimality of $W_0$ we know that $R'$ is not in the 
image of $\overline{NE}(F) \rightarrow \overline{NE}(\overline{X})$.  
So we may assume by relabelling  that $R'$ is
contained in the image of $\overline{NE}(E_0) \rightarrow \overline{NE}(\overline{X})$.

Writing $\nu^*(K_{\overline{\cal F}}+\overline{\Theta}) = K_{\cal G}+\Gamma$ where $\nu\colon E^\nu_0 \rightarrow E_0$
is the normalisation and
$\cal G$ is the restricted foliation
and arguing by induction, and using  (4) of Lemma \ref{lem_tangency},
we may produce a rational curve $\xi$ as required.
\end{proof}

\begin{proof}[Proof of Theorem \ref{thm_cone} and Theorem \ref{thm_prop*}]
By proceeding by induction on the dimension, Lemma \ref{l_prop*}   and Lemma \ref{l_cone} immediately imply  Theorem \ref{thm_cone} and Theorem \ref{thm_prop*}.
\end{proof}

\section{Proof of the Main Theorems}

\begin{proposition}\label{p_main}
Let $X$ be a quasi projective normal variety and let $(X/Z,B)$ be a GLC pair such that $B\ge 0$. Assume that 
\begin{enumerate} 
\item  $(X/Z,B)$ satisfies Property $(*)$ and it is BP stable over $Z$;
\item the induced morphism $f\colon X\to Z$ is a projective equidimensional contraction; and
\item $K_X+B$ is $f$-nef. 
\end{enumerate}

Then the moduli part $M_X$ of $(X/Z,B)$ is nef. 
\end{proposition}

\begin{proof} We first prove the Proposition under the assumption that both $X$ and $Z$ are projective. 
Let $\cal F$ be the foliation on $X$ defined by $f$, let $\Delta$ be the horizontal part of $B$ with respect to $f$ and let $G$ be its vertical part. Then $(\cal F,\Delta)$ satisfies Property $(*)$ and $G$ is a divisor associated to $(\cal F,\Delta)$ (cf. Definition \ref{d_fstar}). Since, by assumption, $f$ is equidimensional,  Proposition \ref{prop_*comparison} implies that the moduli part $M_X$ of  $(X/Z,B)$ is linearly equivalent to $K_{\cal F}+\Delta$ and Proposition \ref{p_canonical} implies that $(\cal F,\Delta)$ is log canonical. 

Assume, by contradiction, that $M_X$ is not nef. Then the Cone theorem (cf. Theorem \ref{thm_cone}) implies that there exists a curve $\xi$ in $X$ which is tangent to $\cal F$ and such that 
\[
M_X\cdot \xi = (K_{\cal F}+\Delta)\cdot \xi<0.
\]
By (3) of Lemma  \ref{lem_tangency}, the curve $\xi$ is vertical with respect to $f$, i.e. $M_X$ is not $f$-nef, which contradicts our assumption. 

\medskip

We now prove the Proposition in general, assuming that $X$ and $Z$ are quasi-projective.
Let $\xi \subset Z$ be a  projective curve. It is enough to show that $M_X\cdot \xi'\ge 0$ for any curve $\xi'$ such that $f(\xi')=\xi$. Let $\Sigma_Z\subset Z$ be the induced divisor. 
Let $\overline{Z}$ be a  projective compactification of $Z$ so that $\Sigma\coloneqq \overline{Z}\setminus Z$ is a divisor. After possibly replacing $\overline Z$ by a log resolution, we may assume that there exists a reduced divisor $\Sigma_{\overline Z}\ge \Sigma$ such that $(\overline Z,\Sigma_{\overline {Z}})$ is log smooth and $\Sigma_{\overline Z}|_Z=\Sigma_Z$. 
Let $\overline{X}$ be a normal projective compactification of $X$ so that we have a morphism $\bar{f}\colon \overline{X} \to \overline{Z}$.
Let $\overline{B}$ denote the closure of $B$, let $\overline{\Delta}$ denote the horizontal part of $\overline{B}$
and let $\overline{\cal F}$ denote the foliation induced by $\bar{f}$.

Let $A$ be an ample divisor on $\overline X$ and let $t>0$ be a rational number. By Proposition \ref{p_prop*_bertini}, there exists an open neighbourhood $V\subset Z$ of the generic point of $\xi$ and  a general $\mathbb Q$-divisor $0\le H\sim tA$ on $\overline{X}$ such that
$(\overline{X}/\overline{Z}, \overline{B}+H)$ satisfies Property $(*)$ over $V$. 
By Theorem \ref{thm_prop*}, there exists a Property $(*)$ modification $\pi\colon Y \rightarrow \overline{X}$ for $(\overline {\cal F},\overline{\Delta})$ so that $\cal G \coloneqq \pi^{-1}\cal F$
is induced by an equidimensional contraction $g\colon Y \rightarrow Z$ and where $\mu\colon Z' \rightarrow \overline{Z}$ is a birational morphism.
Let $C \coloneqq \pi_*^{-1}(\overline{B}+H)+\sum E$ and $\Gamma \coloneqq \pi_*^{-1}(\overline{\Delta}+H)+\sum \epsilon(E)E$, where the sums run over all  the $\pi$-exceptional divisors.
Then
\[(K_{\cal G}+\Gamma)+F = \pi^*(K_{\overline{\cal F}}+\overline{\Delta})\] 
where $F \geq 0$ is supported on 
the complement of $U \coloneqq \pi^{-1}\bar{f}^{-1}(V)$ and, in particular, $F\cdot \xi' \geq 0$ for any curve $\xi''$ on $Y$ such that $\bar{f}\circ \pi(\xi'')=\xi$.
Let $W = \mu^{-1}(V)$. Proposition \ref{prop_*comparison} implies that  the moduli part of $(U/W, C)$ is linearly equivalent to $(K_{\cal G}+\Gamma)\vert_U$,
and since $\pi^*M_X \sim  (K_{\cal G}+\Gamma)\vert_U$, it follows  that $(K_Y+C)\vert_U$ is nef over $W$.

Let $\phi\colon Y \dashrightarrow Y'$ be a 
$(K_Y+C)$-MMP over $Z'$. Note that such a MMP terminates, since $C$ is $g$-big. 
Since $(K_Y+C)|_U$ is nef over $W$, we have that $\phi|_U$ is an isomorphism.
 Proposition \ref{prop_MMP_preserves_*} implies that $(Y'/Z',\phi_*C)$ satisfies Property $(*)$ and that $g$ is equidimensional.   Proposition \ref{prop_*comparison} implies that  the moduli part of $(Y'/Z', \phi_*C)$ is linearly equivalent to $K_{\cal G'} + \phi_*\Gamma$ is nef,
where $\cal G'$ is the induced foliation on $Y'$.
Since $K_{Y'}+\phi_*C$ is nef over $W$, it follows, by the projective case above, that $K_{\cal G'} + \phi_*\Gamma$ is nef. Thus,  $(K_{\cal G}+\Gamma)\vert_U$ is nef, and so 
$(K_{\cal G}+\Gamma+F)\cdot \xi''\geq 0$ for any curve $\xi''$ on $Y$ such that $\bar{f}\circ \pi(\xi'')=\xi$. It follows that $(M_X+tA)\cdot \xi' \geq 0$ for any curve $\xi'$ on $\overline {X}$ such that $\bar{f}(\xi')=\xi$ and for all $t>0$.  Thus, $M_X\cdot \xi'\geq 0$ and the result follows. \end{proof}

\begin{lemma}\label{l_*BP}
Let $X$ be a quasi projective normal variety and let $(X/Z,B)$ be a GLC pair such that $B\ge 0$ and such that the induced morphism $f\colon X\to Z$ is a projective contraction. Assume that 
\begin{enumerate} 
\item $X$ is $\mathbb Q$-factorial and $(X,B)$ is dlt; 
\item  $(X/Z,B)$ satisfies Property $(*)$; and
\item $K_X+B$ is $f$-nef.
\end{enumerate}
Let $f'\colon X'\to Z'$ be a contraction between normal varieties which is birational equivalent to $f$ and such that the induced maps $\alpha\colon Z'\to Z$ and $\beta\colon X'\to X$ are birational morphisms. 
Write $K_{X'}+B'=\beta^*(K_X+B)$ and $\alpha^*(K_Z+B_Z)=K_Z+B'_Z$. Let $B'_{Z'}$ be the discriminant  of $(X'/Z',B')$ on $Z'$.

 Then $B'_Z\le B'_{Z'}$. 
\end{lemma}

\begin{proof}
Let $F$ be a divisor on $Z'$ which is exceptional over $Z$. It is enough to show that 
$m_FB'_Z\le m_FB'_{Z'}$.
 By cutting through general hyperplanes, we may assume  that the centre of $F$ in $Z$ is a closed point $z\in Z$.

We first prove the Lemma under the assumption that $B\ge H$ for a fixed general $f$-ample $\mathbb Q$-divisor $H\ge 0$ on $X$.
By  Proposition \ref{prop_*resolution}, after possibly replacing $Z'$ by a higher model, we may assume that 
 $f'\colon X'\to Z'$ is equidimensional, $X'$ is $\mathbb Q$-factorial and there exists an effective
 $\mathbb R$-divisor $\overline B$ on $X'$ such that
  $(X'/Z',\overline B)$ satisfies Property $(*)$,  $\overline B\ge B'$ 
and the horizontal part of $\overline B-B'$ is $\beta$-exceptional. Note that, by the assumption above, it follows that $\overline B$ is $f'$-big. 
Let $\Sigma_{Z'}$ be the induced divisor on $Z'$. By Lemma \ref{l_star}, $\Sigma_{Z'}$ coincides with the 
 discriminant  $\overline B_{Z'}$
 of $(X'/Z',\overline B)$.
After possibly replacing $Z'$ by a higher model and $\overline B$ by a larger $\mathbb R$-divisor, 
 we may assume that $F$ is a divisor on $Z'$ which is a component of $\Sigma_{Z'}$. 
 Furthermore, since $(X,B)$ is dlt and $B\ge H$, we may assume that there exists a $\mathbb Q$-divisor $0\le D\sim_{\mathbb Q}\overline B$ such that $D\ge H'$ for some ample $\mathbb Q$-divisor $H'$ on $X'$ and $(X,D)$ is klt. Thus, by \cite[Theorem 1.2]{BCHM06}, Theorem \ref{t_main1}, we may run a $(K_{X'}+\overline B)$-MMP over $Z'$.  Let $\phi\colon X'\dashrightarrow Y$ be the induced map, let $C=\phi_*\overline B$ and let $g\colon Y\to Z'$ be the induced morphism. Then Proposition \ref{prop_MMP_preserves_*} implies that $(Y/Z',C)$ satisfies Property $(*)$ and that $g$ is equidimensional.  
Thus, if  $M_Y$ is the moduli part of $(Y/Z',C)$, then as in the proof of Proposition \ref{p_main}, it follows that $M_Y$ is nef. 

Let $p\colon W\to X$ and $q\colon W\to Y$ be birational morphisms which resolve the indeterminacy of the induced map $X\dashrightarrow Y$ and let  
\[
A\coloneqq p^*M_X-q^*M_Y.
\]
By construction and since $K_X+B$ is $f$-nef, it follows that $A$ is vertical with respect to the induced morphism $W\to Z$. 

We claim that $A\ge 0$. Assuming the claim, we first prove that  $m_FB'_{Z} \le m_FB'_{Z'}$. We may write 
\[
K_W+B_W=p^*(K_X+B)\qquad\text{and}\qquad K_W+C_W=q^*(K_Y+C).
\]
Thus, if $h\colon W\to Z'$ is the induced morphism and $C_{Z'}$ is the discriminant of $(Y/Z',C)$, it follows that
\[
A=B_W-C_W+h^*(C_{Z'}-B'_{Z})
\]
and, in particular, for any prime divisor $R$ on $W$ mapping onto $F$, since $m_RC_W=m_FC_{Z'}=1$, we have that 
\[
m_R(B_W + (1-m_FB'_Z)h^*F)=m_R A +m_R C_W\ge 1. 
\]
It follows that $1-m_FB'_Z\ge \gamma_F$, where
\[
\gamma_F=\sup\{t\in \mathbb R\mid (W,B_W+th^*F) \text{ is log canonical over the generic point of }F\}.
\]
Note that $B'_{Z'}$ coincides with the discriminant of $(W/Z,B_W)$ and, in particular, $\gamma_F=1-m_FB'_{Z'}$. 
Thus, $m_FB'_{Z} \le m_FB'_{Z'}$, as claimed. 

\medskip

We now prove the claim. Since $M_Y$ is nef, by the negativity lemma it is enough to prove that $p_*A\ge 0$. 
Let $E\subset X$ be a prime divisor. We want to show that if $E'$ is the strict transform of $E$ in $W$ then $m_{E'}A\ge 0$. 
By Lemma \ref{l_star}, the discriminant $C_{Z'}$ of $(Y/Z',C)$ coincides with the induced divisor $\Sigma_{Z'}$ and, in particular, it is reduced. 
Since $(Z,B_Z)$ is log canonical, it follows that $C_{Z'}\ge B'_{Z}$. Thus, since $A=B_W-C_W+h^*(C_{Z'}-B'_{Z})$, it is enough to show that $m_{E'}B_W\ge m_{E'}C_W$. 
By construction, we may assume that $p(E)\subset \Sigma_X$. Since $(X/Z,B)$ satisfies Property $(*)$, it follows that $m_EB=1$. 
Since $(Y,C)$ is log canonical, it follows that 
\[
m_{E'}B_W=1\ge m_{E'}C_W.
\]
Thus, our claim follows. 

\medskip 

We now prove the Lemma in general. Let $H$ be an ample divisor on $X$. After possibly replacing $Z$ by a smaller open neighbourhood of $z$,  Proposition \ref{p_prop*_bertini} implies that, if 
$H$ sufficiently general then $(X/Z,B+tH)$ satisfies Property $(*)$ for any sufficiently small rational number $t>0$. Let $D=B+tH$. Note that, since $H$ is general,  the discriminant of $(X/Z,D)$ coincides with the discriminant of $(X/Z,B)$ and the argument above implies that if $K_{X'}+D'=\beta^*(K_X+D)$ and $D'_{Z'}$ is the discriminant of $(X'/Z',D')$ then $m_FD'_{Z'}\ge m_F B'_{Z}$. 
Note that $D'=B'+tf^*H$ and, therefore, we have that $\lim_{t\to 0}m_FD'_{Z'}=m_FB'_{Z'}$. Thus, $m_FB'_{Z'}\ge m_FB'_Z$ and 
the Lemma follows.
\end{proof}

\begin{theorem}
\label{t_*BP}
Let $X$ be a quasi projective normal variety and let $(X/Z,B)$ be a GLC pair such that $B\ge 0$ and such that the induced morphism $f\colon X\to Z$ is an equidimensional projective contraction. Let $\cal F$ be the foliation induced by $f$ and let $\Delta$ be the horizontal part of $B$. 
Suppose  that 
\begin{enumerate}
\item $(X/Z,B)$ satisfies Property $(*)$; 
\item $(\cal F,\Delta)$ is log canonical; and
\item $K_X+B$ is $f$-nef.
\end{enumerate}

Then  $(X/Z,B)$ is BP stable over $Z$. 
\end{theorem}

\begin{proof}
Let $\Sigma_Z$ be the divisor induced on $Z$. 
By Lemma \ref{l_star}, $\Sigma_Z$ coincides with the discriminant $B_Z$  of $(X/Z,B)$. 
Let $F$ be a divisorial valuation over $Z$ and let $f'\colon X'\to Z'$ be a contraction between normal varieties which is birational equivalent to $f$ such that the induced maps $\alpha\colon Z'\to Z$ and $\beta\colon X'\to X$ are birational morphisms and such that $F$ is a divisor on $Z'$. 
Write $K_{X'}+B'=\beta^*(K_X+B)$ and let $B'_{Z'}$ be the discriminant  of $(X'/Z',B')$ on $Z'$.
Since $(\cal F,\Delta)$ is log canonical, Proposition \ref{p_canonical} implies that $(X/Z,B)$ is BP semi-stable over $Z$. Thus,
 it is enough to show that if $\alpha^*(K_Z+B_Z)=K_Z+B'_Z$ then $m_FB'_Z\le m_FB'_{Z'}$.  

Let $\mu\colon \overline X\to X$ be a $\mathbb Q$-factorial dlt modification of $(X,B)$ (e.g. see \cite[Theorem 1.34]{Kollar13}) and let $\overline f\colon \overline X\to Z$ be the induced morphism. 
We may write 
\[
K_{\overline X}+\overline B=\mu^*(K_X+B).
\]
Then $(\overline X/Z,\overline B)$ satisfies Property $(*)$, $K_{\overline X}+\overline  B$ if $\overline f$-nef and, by (1) of Remark \ref{r_induced}, it follows that the discriminant of $(\overline X/Z,\overline B)$ coincides with the discriminant of $(X/Z,B)$ (note that $\overline f$ is not necessarily equidimensional). Thus, Lemma \ref{l_*BP} implies our claim. 
\end{proof}

\begin{theorem} 
\label{thm_main}
Assume termination of log canonical flips in dimension $n$. 

Let $X$ be a quasi projective normal variety of dimension $n$ and let 
$(X/Z,B)$ be a GLC pair  such that    $B\ge 0$ and such that the induced morphism $f\colon X\to Z$ is a projective contraction.  Assume that 
\begin{enumerate} 
\item  $(X/Z,B)$ satisfies Property $(*)$ and it is BP stable over $Z$; and 
\item $K_X+B$ is $f$-nef. 
\end{enumerate}

Then the moduli part $M_X$ of $(X/Z,B)$ is nef. 
\end{theorem}

\begin{proof}
By  Proposition \ref{prop_*resolution}
there exist an equidimensional 
contraction $f'\colon X'\to Z'$ between normal varieties
which is birationally equivalent to $f$    and a $\mathbb R$-divisor $\overline B\ge 0$ on $X'$ such that 
 the induced maps $\alpha\colon Z'\to Z$ and $\beta\colon X'\to X$ are birational morphisms,
$X'$ is $\mathbb Q$-factorial,  $(X'/Z',\overline B)$ satisfies Property $(*)$, 
and if 
\[
K_{X'}+B' = \beta^*(K_X+B),
\] 
then $\overline B\ge B'$ and the horizontal part of $\overline B-B'$ is $\beta$-exceptional. 
Moreover, by Lemma \ref{lem_toroidal_lc}, we may assume that if $\cal F'$ is the foliation induced by $f'$ and $\overline \Delta$ is the horizontal part of $\overline B$ then $(\cal F',\overline \Delta)$ is log canonical. 

Let $\phi\colon X'\dashrightarrow Y$ be a $(K_{X'}+\overline B)$-MMP over $Z'$. Note that such a MMP terminates, as we are assuming termination of flips in dimension $n$. 
 Let $C=\phi_*\overline B$ and let $g\colon Y\to Z'$ be the induced morphism. Then Proposition \ref{prop_MMP_preserves_*} implies that $(Y/Z',C)$ satisfies Property $(*)$ and that $g$ is equidimensional.  
 Note that, by Proposition \ref{prop_*comparison}, it follows that $\phi$ is also a $(K_{\cal F'}+\overline \Delta)$-MMP  over $Z'$ and, in particular, if $\cal G$ is the foliation induced by $g$ and $\Gamma$ is the horizontal part of $C$, it follows that $(\cal G,\Gamma)$ is log canonical. 
 Thus, if  $M_Y$ is the moduli part of $(Y/Z',C)$, then Proposition \ref{p_main}  implies that $M_Y$ is nef and
Theorem \ref{t_*BP} implies that   $(Y/Z',C)$ is BP stable over $Z'$.
By construction and since $K_X+B$ is $f$-nef, it follows that $(X,B)$ and $(Y,C)$ are crepant over the generic point of $Z$. Let $p\colon W\to X$ and $q\colon W\to Y$ be birational morphisms that resolve 
the indeterminacy of the induced map $X\dashrightarrow Y$.

We claim that $p^*M_X\sim q^*M_Y$. Assuming the claim, note that since $M_Y$ is nef, it follows that also $M_X$ is nef. Thus, the Theorem follows. 
To prove the claim, after possibly replacing $(X/Z,B)$ by $(X/Z,B+f^*A)$ and $(Y/Z',C)$ by $(Y/Z',C+g^*\alpha^*A)$ for a sufficiently ample divisor $A$ on $Z$, we may assume that $M_X$ is nef. 
Thus, Proposition \ref{p_maximal} implies the claim. 
 \end{proof}

\begin{proof}[Proof of Theorem \ref{t_main1}] 
By  Proposition \ref{prop_*resolution}
there exist an equidimensional 
contraction $f'\colon X'\to Z'$ between normal varieties
which is birational equivalent to $f$    and a $\mathbb R$-divisor $\overline B\ge 0$ on $X''$ such that 
 the induced maps $\alpha\colon Z'\to Z$ and $\beta\colon X'\to X$ are birational morphisms,
$X''$ is $\mathbb Q$-factorial,  $(X'/Z',\overline B)$ satisfies Property $(*)$, 
and if 
\[
K_{X'}+B' = \beta^*(K_X+B),
\] 
then $\overline B\ge B'$ and the horizontal part of $\overline B-B'$ is $\beta$-exceptional.

Let $\phi\colon X'\dashrightarrow Y$ be a $(K_{X'}+\overline B)$-MMP over $Z$. Note that such a MMP terminates, as we are assuming termination of flips in dimension $n$. 
 Let $C=\phi_*\overline B$ and let $g\colon Y\to Z'$ be the induced morphism. Then Proposition \ref{prop_MMP_preserves_*} implies that $(Y/Z',C)$ satisfies Property $(*)$ and that $g$ is equidimensional. 
By construction and since $K_X+B$ is $f$-nef, it follows that (1) holds.  
As in the proof of Theorem \ref{thm_main}, it follows that if $\cal G$ is the foliation induced by $g$ and $\Gamma$ is the horizontal part of $C$ then $(\cal G,\Gamma)$ is log canonical. Thus, Theorem \ref{t_*BP} and Theorem \ref{t_logstable} imply (3) and Theorem \ref{thm_main} implies (2). Finally, Proposition \ref{p_maximal} implies (4).
\end{proof}

\begin{remark}
Note that, by \cite[Theorem 1.4]{HH20} (see also \cite[Theorem 0.2]{Lai11}),
 Theorem \ref{t_main1} and Theorem \ref{thm_main} hold without assuming termination of flips, if we assume that the relative dimension of the morphism $X\to Z$ is at most three.
\end{remark}

\section{Applications}

\subsection{Pseudo-effectivity of relative canonical divisor}

\begin{lemma}
\label{lem_psef_KF}
Let $(\cal F,\Delta)$ be a foliated pair on  a normal variety 
$X$, where $\cal F$ is the foliation induced by a contraction $f\colon X \to Y$ between normal varieties such that $\dim Z < \dim X$ and $\Delta\ge 0$. Let $F$ be the general fibre of $f$. Suppose that $(\cal F,\Delta)$ is log canonical over the generic point of $Z$ and that $(K_{\cal F}+\Delta)|_F$ is pseudo-effective. 

Then $K_{\cal F}+\Delta$ is pseudo-effective.
\end{lemma}
\begin{proof}
Assume be contradiction that $K_{\cal F}+\Delta$ is not pseudo-effective. 
By Theorem \ref{thm_prop*}, after possibly taking a Property $(*)$ modification, we may assume that $(\cal F,\Delta)$ satisfies Property $(*)$ and that $X$ is klt.  Let $G$ be the associated divisor to $(\cal F,\Delta)$. Let  $\phi\colon X \dashrightarrow X'$ be a $(K_{X}+\Delta+G)$-MMP over $Z$ with scaling of a sufficiently ample divisor $A$ on $X$.
This will terminate in a relative Mori fibre space $g\colon X' \rightarrow W$ over $Z$,
which will restrict to a Mori fibre space structure on $F' \coloneqq \phi_*F$, contrary to our hypothesis.
\end{proof}

\begin{theorem}
\label{thm_relatively_pseudo_eff}
Let $(X/Z, B)$ be a GLC pair 
and let $F$ be a general fibre of the induced contraction $f\colon X \rightarrow Z$. 
Suppose that $(K_X+B)\vert_F$ is pseudo-effective.  

Then the moduli part $M_X$ of $(X/Z, B)$ is pseudo-effective.
In particular, $K_{X/Z}+B$ is pseudo-effective.
\end{theorem}
\begin{proof}
By  Proposition \ref{prop_*resolution}
there exist a commutative diagram 
\[
\begin{tikzcd}
  X' \arrow[r, "\beta"] \arrow[d, "f'"'] & X \arrow[d, "f"] \\
  Z' \arrow[r, "\alpha"] &  Z
 \end{tikzcd}
\]
where $\alpha$ and $\beta$ are birational and a $\mathbb R$-divisor $\overline B\ge 0$ on $X'$ such that 
$X'$ is $\mathbb Q$-factorial, $f'$ is equidimensional, $(X'/Z',\overline B)$ satisfies Property $(*)$ 
and if 
\[
K_{X'}+B' = \beta^*(K_X+B),
\] 
then $\overline B\ge B'$ and the horizontal part of $\overline B-B'$ is $\beta$-exceptional. 

Let $\cal F'$ be the foliation induced by $X' \rightarrow Z'$ and let $\Delta$ denote the horizontal part of $\overline B$
so that, by Proposition \ref{prop_*comparison},
the moduli part $M_{X'}$ of $(X'/Z', \overline B)$ is equal to $K_{\cal F'}+\Delta$.
Observe that $(\cal F', \Delta)$ is log canonical above the generic point of $Z'$ and so by Lemma \ref{lem_psef_KF}
we have that $K_{\cal F'}+\Delta = M_{X'}$ is pseudo-effective. 
It follows that $M_X = \beta_*M_{X'}$ is also pseudo-effective.
\end{proof}

 Note that, assuming termination of log canonical flips,  the same result follows from Theorem \ref{thm_main} after running a $(K_X+B)$-MMP over $Z$.

\subsection{Comparing pseudo-effectivity of log canonical divisors}

The following result should be compared with \cite[Theorem 1.3]{CP19}.  We remark that the invariant divisor $D$ appearing in the 
statement of Theorem \ref{thm_KF_KX} is important for certain applications.

\begin{theorem}
\label{thm_KF_KX}
Let $X$ be a klt projective 
variety and let $(\cal F,\Delta)$ be a foliated pair such that $K_{\cal F}+\Delta$ is not pseudo-effective.
Let $D$ be an $\cal F$-invariant $\mathbb R$-divisor on $X$ so that
$(X, \Delta+D)$ is log canonical.

Then $K_X+\Delta+D$ is not pseudo-effective.  Moreover, $X$ is covered by $(K_X+\Delta+D)$-negative curves
tangent to $\mathcal F$, and there exists a $(K_X+\Delta+D)$-negative rational curve tangent to $\cal F$.
\end{theorem}
\begin{proof}
After possibly replacing $X$ by a small $\mathbb Q$-factorialisation, we may assume that $X$ is $\mathbb Q$-factorial.

Let $\alpha$ be a movable class so that $(K_{\cal F}+\Delta)\cdot \alpha<0$.
Let $0 = \cal E_0 \subset \cal E_1 \subset \dots \subset \cal E_r = T_{\cal F}$
be the Harder-Narasimhan filtration with respect to $\alpha$.
By \cite[Corollary 2.20]{AD19} it follows that there exists a foliation $\cal G$ on $X$
such that $T_{\cal G} = \cal E_1$ and such that $\cal G$ has algebraic leaves.
Since $\mu_\alpha(T_{\cal G}) \geq \mu_{\alpha}(T_{\cal F})$ we may freely replace $\cal F$ by $\cal G$
and so we may assume that $\cal F$ has algebraic leaves.

Let  $p\colon X' \rightarrow X$
be a Property $(*)$ modification of $(\cal F, \Delta)$, whose existence is guaranteed by Theorem \ref{thm_prop*}
 and such that $\cal F' = \pi^{-1}\cal F$ is induced by an equidimensional contraction $f'\colon X' \rightarrow Z$.
Let $\Delta' = \pi_*^{-1}\Delta'+\sum \epsilon(E)E$ 
where the sum runs over the $p$-exceptional divisors
so that we may write
\[(K_{\cal F'}+\Delta')+F = p^*(K_{\cal F}+\Delta)\]
and
\[K_{X'}+\Delta'+D' = p^*(K_X+\Delta+D)+R\]
where $F \geq 0$, $D'$ has invariant support (but is not necessarily effective) 
and $R \geq 0$ is supported on exceptional divisors which are not $\cal F'$-invariant.

Let $G \geq 0$ be a divisor  on $X'$ associated to $(\cal F',\Delta')$ (cf. Definition \ref{d_fstar}).
As in the proof of 
Theorem \ref{thm_prop*}, we may run a $(K_{\cal F'}+\Delta')$-MMP with scaling of some ample divisor $A$
by running a $(K_{X'}+\Delta'+G)$-MMP over $Z$ with scaling of $A$.  
Call this MMP
$\phi\colon X' \dashrightarrow X''$ and observe that since $K_{\cal F'}+\Delta'$ is not pseudo-effective this MMP must terminate
in a Mori fibre space
$g\colon X'' \rightarrow Y$ over $Z$.
Denote by $f''\colon X'' \rightarrow Z$ the induced morphism. 
For a general curve $\Sigma$ contracted by $g$ we have that $\phi_*D'\cdot \Sigma = 0$, and 
$(K_{X''}+\phi_*\Delta')\cdot \Sigma = K_{\cal F''}\cdot \Sigma$, 
where $\cal F'' $ is the induced foliation on $X''$.
It follows that $K_{X''}+\phi_*(\Delta'+D')$ is not pseudo-effective,
hence the same is true for $K_X+\Delta+D$, as required.

To prove our final claims let $X'_z$ (resp. $X''_z$) be the fibre of $f'$ (resp. $f'$) over a general point $z \in Z$.
Let $X_z = p(X'_z)$.  It suffices to show that $X'_z$ is covered by $(K_{X'}+\Delta'+D')$-negative curves and that there exists 
at least one $(K_{X'}+\Delta'+D')$-negative rational curve contained in $X'_z$.
This first claim is an immediate consequence of the fact that a general complete intersection curve $C \subset X''_z$ will be disjoint
from $\phi(\Exc \phi) \cap X''_z$, and the fact that  
$-(K_{X''}+\phi_*(\Delta'+D'))\vert_{X''_z} \sim_{\mathbb R} -(K_{X''_z}+\phi_*\Delta'\vert_{X''_z})$ is ample. 
The second claim follows by considering  the first extremal contraction in the rational map $g\circ \phi\colon X'\dashrightarrow Y$ which
is not an isomorphism over the generic point of $Z$, as this
will produce the required rational curve.
\end{proof}

\begin{proof}[Proof of Theorem \ref{t_main3}] Theorem \ref{t_main3} follows immediately from Theorem \ref{thm_KF_KX}.
\end{proof}

\subsection{$b$-nefness of the moduli part in the $f$-trivial case}
\label{s_bnefness}

\begin{proof}[Proof of Theorem \ref{t_f-trivial}]
By Remark \ref{r_induced},    after possibly replacing $(X,B)$ by a dlt modification (e.g. see \cite[Theorem 1.34]{Kollar13}), we may assume that $X$ is $\mathbb Q$-factorial and $(X,B)$ is dlt. 
By  Proposition \ref{prop_*resolution},
there exist a commutative diagram 
\[
\begin{tikzcd}
  X' \arrow[r, "\beta"] \arrow[d, "f'"'] & X \arrow[d, "f"] \\
  Z' \arrow[r, "\alpha"] &  Z
 \end{tikzcd}
\]
where $\alpha$ and $\beta$ are birational morphisms and a $\mathbb R$-divisor $\overline B\ge 0$ on $X'$ such that 
$X'$ is $\mathbb Q$-factorial, $f'$ is equidimensional, $(X'/Z',\overline B)$ satisfies Property $(*)$,  
and if 
\[
K_{X'}+B' = \beta^*(K_X+B),
\] 
then $\overline B\ge B'$ and the horizontal part of $\overline B-B'$ is $\beta$-exceptional. 
Since $(X/Z,B)$ is BP stable over $Z$, it follows that $\beta^*M_X=M_{X'}$ is the moduli part of $(X'/Z',B')$. In particular, it is enough to show that $M_{X'}$ is nef. 
Let $\Sigma_{Z'}$ be the discriminant of $(X'/Z',\overline B)$, let $B'_{Z'}$ be the discriminant of $(X'/Z',B')$  and let $B''\coloneqq B'+f'^*( \Sigma_{Z'}-B'_{Z'})$. Then 
$(X'/Z',B'')$ satisfies Property $(*)$, $K_{X'}+B''\sim_{f',\mathbb R} 0$, the discriminant of $(X'/Z',B'')$ coincides with $\Sigma_{Z'}$ and its moduli part is equal to $M_{X'}$.

Let $A$ be an ample divisor on $X$ and let $t>0$. Since $(X,B)$ is dlt, there exist a $\mathbb R$-divisor $D\ge 0$ and an ample $\mathbb Q$-divisor $H\ge 0$ on $X$ such that $(X,D+H)$ is klt and $B+tA\sim_{\mathbb R}D+H$.  In particular, we may run a $(K_{X'}+\overline B)$-MMP over $Z'$ with scaling of $\beta^*A$. Note that, since $\overline B-B''$ is horizontal with respect to $f'$ and $\beta$-exceptional and since $K_{X'}+B''\sim_{f',\mathbb R} 0$, it follows that after finitely many steps of this MMP, we obtain a birational map $\phi\colon X'\dashrightarrow Y$ over $Z'$ which contracts the support of $\overline B-B''$. 
In particular, if  $C\coloneqq \phi_*\overline B$, then we also have that $C=\phi_*B''$.  Let $g\colon Y\to Z'$ be the induced morphism. Then Proposition \ref{prop_MMP_preserves_*} implies that $(Y/Z',C)$ satisfies Property $(*)$ and that $g$ is equidimensional.  Let $\cal G$ be the foliation induced by $g$ and let $\Delta$ be the horizontal part of $C$. Then, as in the proof of Theorem \ref{thm_main}, it follows that $(\cal G,\Delta)$ is log canonical. Thus, Theorem \ref{t_*BP} implies that $(Y/Z',C)$ is BP stable over $Z'$ and Proposition \ref{p_main} implies that the moduli part $M_Y$ of $(Y/Z', C)$ is nef. Thus, Proposition \ref{p_maximal} and Remark \ref{r_maximal} imply that $M_{X'}$ is nef, as claimed.  
\end{proof}


%

\subsection{A counterexample to the semi-ampleness of the moduli part}
\label{s_counterexample_semi_ample}
In contrast to the $f$-trivial case, in general we cannot hope for the semi-ampleness
of the moduli part, as the following examples show.  These examples also show that the 
abundance conjecture does not hold for algebraically integrable foliations. 

Both of the following examples are adapted from \cite[Section 3]{Keel99}.

\begin{example}
Let $C$ be a curve of genus $g\geq 2$, let $X = C \times C$, let $\Delta$ denote the diagonal and let $\pi\colon X \rightarrow C$
be the projection onto the first coordinate.
Then $(X/Z,\Delta)$ satisfies Property $(*)$ with discriminant $B_C=0$. Thus, the moduli part of $(X/Z,\Delta)$ is $M_X=K_{X/C}+\Delta$. 
But \cite[Theorem 3]{Keel99} implies that $M_X$ is big and nef, but not semi-ample. Clearly, in this case, the pair  $(X,\Delta)$ is log canonical but not klt. 
\end{example}

\begin{example}
We continue to use the same notation as in the above example, Let $E$ be an elliptic curve and let  $P \in E$ be a closed point.  Let $Y = C \times E$ and let $S$ be the family of curves obtained
by gluing $Y$ to $X$ along the curves $C \times \{P\}$ and $\Delta$.  This gives a family of stable curves $S \rightarrow C$ 
which corresponds
to a morphism $C \rightarrow \overline{\cal M}_{g}$.  Let $\cal U \rightarrow \overline{\cal M}_g$ denote the universal family.  

We may find a projective surface $T$ and a generically finite morphism $f\colon T \rightarrow \overline{\cal M}_g$ so that $f(T)$ contains the image $\overline C$ of $C$.
Let $Z = T \times_{\overline{\cal M}_g} \cal U$ and observe that $Z$ is a projective threefold.
Let $p\colon Z \rightarrow T$ be the induced morphism.  Observe that $p$ is semi-stable and that $K_Z$ is relatively nef over $Z$. Thus, if $B=0$ then $(Z/T,B)$ satisfies the Property $(*)$ with discriminant $B_T=0$ and, by Theorem \ref{t_*BP}, it is BP stable over $Z$. 
It follows that $K_{Z/T}$ is the moduli part of
$(Z/T, B)$. 
Let $C'$ be an irreducible component of $f^{-1}(\overline C)$ and let $S' = p^{-1}(C')$. Observe that $S' = X' \cup Y'$ where 
$\sigma\colon X' \rightarrow X$ is 
finite and
ramified only along fibres of $X' \rightarrow C'$
It follows that $K_{Z/T}\vert_{X'} = K_{X'/C'}+\Delta' = \sigma^*(K_{X/C}+\Delta)$ where $\Delta'$ is the graph of $C' \rightarrow C$.  
In particular, $K_{Z/T}$ is not semi-ample.
\end{example}

%
%

One might instead ask if something weaker holds true, namely, if $\nu(K_{\cal F}) = \kappa(K_{\cal F})$ 
for any algebraically integrable foliation $\cal F$ with canonical singularities and such that $K_{\cal F}$ is nef.  In light of Section \ref{s_bnefness}
any such counter example must have  $1 \leq \nu(K_{\cal F}) \leq \dim X -1$.

%
%
%
%

\bibliography{math.bib}

\begin{thebibliography}{BCHM10}

\bibitem[AD14]{AD14}
C.~Araujo and S.~Druel.
\newblock On codimension 1 del {P}ezzo foliations on varieties with mild
  singularities.
\newblock {\em Math. Ann.}, 360(3-4):769--798, 2014.

\bibitem[AD19]{AD19}
C.~Araujo and S.~Druel.
\newblock Characterization of generic projective space bundles and algebraicity
  of foliations.
\newblock {\em Comment. Math. Helv.}, 94(4):833--853, 2019.

\bibitem[AK00]{AK00}
A.~Abramovich and K.~Karu.
\newblock Weak semistable reduction in characteristic 0.
\newblock {\em Invent. Math.}, 139(2):241--273, 2000.

\bibitem[Amb98]{Ambro98}
F.~Ambro.
\newblock The locus of log canonical singularities.
\newblock 1998.
\newblock arXiv:9806067.

\bibitem[Amb04]{Ambro04}
F.~Ambro.
\newblock Shokurov's boundary property.
\newblock {\em J. Differential Geom.}, 67(2):229--255, 2004.

\bibitem[Amb11]{Ambro11}
F.~Ambro.
\newblock Basic properties of log canonical centers.
\newblock In {\em Classification of algebraic varieties}, EMS Ser. Congr. Rep.,
  pages 39--48. Eur. Math. Soc., Z\"{u}rich, 2011.

\bibitem[BCHM10]{BCHM06}
C.~Birkar, P.~Cascini, C.~Hacon, and J.~M\textsuperscript{c}Kernan.
\newblock Existence of minimal models for varieties of log general type.
\newblock {\em J. Amer. Math. Soc.}, 23(2):405--468, 2010.

\bibitem[BM16]{bm16}
F.~Bogomolov and M.~McQuillan.
\newblock Rational curves on foliated varieties.
\newblock In {\em Foliation theory in algebraic geometry}, Simons Symp., pages
  21--51. Springer, Cham, 2016.

\bibitem[CP19]{CP19}
F.~Campana and M.~P\u{a}un.
\newblock Foliations with positive slopes and birational stability of orbifold
  cotangent bundles.
\newblock {\em Publ. Math. Inst. Hautes \'{E}tudes Sci.}, 129:1--49, 2019.

\bibitem[CS20]{CS20}
P.~Cascini and C.~Spicer.
\newblock On the {MMP} for rank one foliations on threefolds.
\newblock 2020.
\newblock arxiv:2012.11433.

\bibitem[CS21]{CS21}
P.~Cascini and C.~Spicer.
\newblock M{MP} for co-rank one foliations on threefolds.
\newblock {\em Invent. Math.}, 225(2):603--690, 2021.

\bibitem[Dru17]{Druel15b}
S.~Druel.
\newblock On foliations with nef anti-canonical bundle.
\newblock {\em Trans. Amer. Math. Soc.}, 369(11):7765--7787, 2017.

\bibitem[Dru21]{Druel21}
S.~Druel.
\newblock Codimension 1 foliations with numerically trivial canonical class on
  singular spaces.
\newblock {\em Duke Math. J.}, 170(1):95--203, 2021.

\bibitem[HH20]{HH20}
K.~Hashizume and Z.-Y. Hu.
\newblock On minimal model theory for log abundant lc pairs.
\newblock {\em J. Reine Angew. Math.}, 767:109--159, 2020.

\bibitem[Kaw98]{Kawamata98}
Y.~Kawamata.
\newblock Subadjunction of log canonical divisors. {I}{I}.
\newblock {\em Amer. J. Math.}, 120(5):893--899, 1998.

\bibitem[Kee99]{Keel99}
S.~Keel.
\newblock Basepoint freeness for nef and big line bundles in positive
  characteristic.
\newblock {\em Ann. of Math. (2)}, 149(1):253--286, 1999.

\bibitem[KM98]{KM98}
J.~Koll{\'a}r and S.~Mori.
\newblock {\em Birational geometry of algebraic varieties}, volume 134 of {\em
  Cambridge tracts in mathematics}.
\newblock Cambridge University Press, 1998.

\bibitem[KMM92]{KMM92c}
J.~Koll{\'a}r, Y.~Miyaoka, and S.~Mori.
\newblock Rational connectedness and boundedness of {Fano} manifolds.
\newblock {\em J. Diff. Geom.}, 36:765--779, 1992.

\bibitem[Kol07]{Kollar07}
J.~Koll{\'a}r.
\newblock Kodaira's canonical bundle formula and adjunction.
\newblock In {\em Flips for 3-folds and 4-folds}, volume~35 of {\em Oxford
  Lecture Ser. Math. Appl.}, pages 134--162. Oxford Univ. Press, Oxford, 2007.

\bibitem[Kol13]{Kollar13}
J.~Koll{\'a}r.
\newblock {\em Singularities of the minimal model program}, volume 200 of {\em
  Cambridge Tracts in Mathematics}.
\newblock Cambridge University Press, Cambridge, 2013.
\newblock With a collaboration of S{\'a}ndor Kov{\'a}cs.

\bibitem[Lai11]{Lai11}
C.-J. Lai.
\newblock Varieties fibered by good minimal models.
\newblock {\em Math. Ann.}, 350(3):533--547, 2011.

\bibitem[Miy93]{Miyaoka93}
Y.~Miyaoka.
\newblock Relative deformations of morphisms and applications to fibre spaces.
\newblock {\em Comment. Math. Univ. St. Paul.}, 42(1):1--7, 1993.

\bibitem[PS09]{PS09}
Y.~Prokhorov and V.~Shokurov.
\newblock Towards the second main theorem on complements.
\newblock {\em J. Algebraic Geom.}, 18(1):151--199, 2009.

\bibitem[Sho13]{Shokurov13}
V.~V. Shokurov.
\newblock Log adjunction: effectiveness and positivity, 2013.
\newblock arXiv:1308.5160.

\bibitem[Sho20]{Shokurov20}
V.~V. Shokurov.
\newblock Log adjunction: moduli part, 2020.
\newblock arXiv:2111.01310.

\bibitem[Spi20]{Spicer20}
C.~Spicer.
\newblock Higher-dimensional foliated {M}ori theory.
\newblock {\em Compos. Math.}, 156(1):1--38, 2020.

\end{thebibliography}
\bibliographystyle{alpha}

\end{document}